\documentclass{amsart}
\usepackage{amsmath,amsfonts,amssymb,amsthm,mathtools,tensor}
\usepackage[a4paper, margin=1in]{geometry} 

\usepackage{hyperref}
\usepackage{url}
\usepackage{xcolor}

\newcommand{\inner}[2]{\left\langle #1, #2 \right\rangle}

\DeclareMathOperator*{\argmin}{arg\,min}

\title{Convergence of an actor-critic gradient flow for entropy regularised MDPs in general spaces}

\addtocounter{footnote}{0}
\author{Denis Zorba}
\address{School of Mathematics, University of Edinburgh, UK}
\email{ezorba@ed.ac.uk}

\author{David \v{S}i\v{s}ka}
\address{School of Mathematics, University of Edinburgh, UK}
\email{d.siska@ed.ac.uk}

\author{Lukasz Szpruch}
\address{School of Mathematics, University of Edinburgh, UK, 
The Alan Turing Institute, UK, and Simtopia, UK}
\email{l.szpruch@ed.ac.uk}

\newtheorem{lemma}{Lemma}[section]
\newtheorem{definition}{Definition}[section]
\newtheorem{assumption}{Assumption}[section]
\newtheorem{theorem}{Theorem}[section]
\newtheorem{corollary}{Corollary}[section]
\newtheorem{remark}{Remark}[section]

\begin{document}


\maketitle

\begin{abstract}
We prove the stability and global convergence of a coupled actor-critic gradient flow for infinite-horizon and entropy-regularised Markov decision processes (MDPs) in continuous state and action space with linear function approximation under Q-function realisability.
We consider a version of the actor critic gradient flow where the critic is updated using temporal difference (TD) learning while the policy is updated using a policy mirror descent method on a separate timescale.
For general action spaces, the relative entropy regularizer is unbounded and thus it is not clear a priori that the actor-critc flow does not suffer from finite-time blow-up.
Therefore we first demonstrate stability which in turn enables us obtain a convergence rate of the actor critic flow to the optimal regularised value function.
The arguments presented show that timescale separation is crucial for stability and convergence in this setting.
\end{abstract}

\section{Introduction}
In reinforcement learning (RL) an agent aims to learn an optimal policy that maximizes the expected cumulative reward through repeated interactions with its environment.
Such methods typically involve two key components: policy evaluation and policy improvement.
During policy evaluation, the advantage function corresponding to a policy, or its function approximation, is updated using state, action and reward data generated under this policy.
Policy improvement then uses this approximate advantage function to update the policy, most commonly through some policy gradient method.
Algorithms that explicitly combine these two components are known as actor-critic (AC) methods \cite{konda}, where the actor corresponds to policy improvement and the critic to policy evaluation.

There are many policy gradient methods to choose from.
In the last decade trust region policy optimization (TRPO) methods~\cite{schulman2015trust} and methods inspired by these like PPO~\cite{schulman2017proximal} have become increasingly well-established due to their impressive empirical performance. 
Largely, this is because they alleviate the difficulty in choosing appropriate step sizes for the policy gradient updates: for vanilla policy gradient even a small change in the parameter may result in large change in the policy, leading to instability, but TRPO prevents this by explicitly ensuring the KL divergence between successive updates is smaller than some tolerance. 
Mirror descent replaces the TRPO's hard constraint with a penalty leading to a first order method which is also ameanable to analysis.
Indeed, at least for direct parametrization, it is known to converge with sub-linear and even linear rate for entropy regularised problems (depending on exact assumptions)~\cite{lan, lan2023policy,david_fisher}.  

Due to the favourable analytical properties of mirror descent, in this paper we consider a version of the actor critic gradient flow where the policy is updated using a policy mirror descent method 
while
the critic is updated using temporal difference (TD) on a separate timescale. 

Entropy-regularised MDPs are widely used in practice since the entropic regularizer leads to a number of desirable properties: it has a natural interpretation as something that drives exploration, 
it ensures that there is a unique optimal policy and it can accelerate convergence of policy gradient methods ~\cite{mei2020global}.
However, analysing the stability and convergence of actor-critic methods in this entropy-regularised setting with general state and action spaces remains highly non-trivial due to lack of a priori bounds on the value functions.

To address the actor critic methods for entropy regularised MDPs in general action spaces, a careful treatment of tools from two timescale analysis, convex analysis over both Euclidean spaces and measure spaces must be deployed. 

In this paper, we address precisely this challenge. We study the stability and convergence of a widely used actor-critic algorithm in which the critic is updated using Temporal Difference (TD) learning \cite{sutton1988learning}, and the policy is updated through Policy Mirror Descent \cite{lan}. Our analysis employs a two-timescale update scheme \cite{borkar1997actor}, where both the actor and critic are updated at each iteration with the critic updated on a faster timescale.

\subsection{Related works}

We focus on the subset of RL literature that address the convergence of coupled actor-critic algorithms.
In the unregularised setting, actor-critic methods have been studied extensively. The first convergence results in the two-timescale regime established asymptotic convergence in the continuous-time limit of coupled updates (\cite{borkar1997actor,konda}). Most modern research employs linear function approximation for the critic, where linear convergence rates have been obtained under various assumptions on the step-sizes of the actor and critic (\cite{barakat22a-linear,linear2,stochastic_approx_application_AC}).

Closely related to our work is \cite{jordan}, which considers the same two-timescale actor-critic scheme in the continuous-time limit for unregularised MDPs, with an overparameterized neural network used for the critic. However, convergence to the optimal policy was not established, and a restarting mechanism was required to ensure the stability of the dynamics.

In the entropy-regularised setting, \cite{cayci,cayci_neural} address the convergence of a natural actor critic algorithm. However, the convergence and stability of these results rely on the finite cardinality of the action space in presence of entropy regularisation.

\subsection{Our Contribution}

Under linear $Q^{\pi}_{\tau}$-realisability assumption, we address the following question:

\textit{``Is the actor-critic gradient flow for entropy-regularised MDPs in general action spaces stable and convergent, and if so, at what rate?''}

There are two main technical challenges one has to overcome when working with entropy-regularised MDPs in general action spaces.

\begin{itemize}
\item Even in mirror descent with exact advantage, the rate of convergence depends on a constant term $\int_S \operatorname{KL}(\pi^*(\cdot|s)|\pi_0(\cdot|s)d^{\pi^\ast}_\rho(ds)$. 
See~\cite{lan2023policy, david_fisher}.
In general action spaces, without entropy regularisation it is almost impossible to choose $\pi_0$ which would make this term finite, see Remark~\ref{remark:entropy_action}.
Thus we need to include the regularisation in the analysis.

\item Moreover, ensuring that the relative entropy does not blow up is difficult in general action spaces.
In the finite action space setting, for any measure $\mu \in \mathcal{P}(A)$ such that $\mu(a_i) > 0$ and for all $s \in S$ it holds that $\operatorname{KL}(\pi(\cdot|s) |\mu) \leq \log |A|$. 
In general action spaces the $\operatorname{KL}$ divergence has no  upper bound (can be $+\infty$) even if $\mu$ has full support. 
Under mild assumptions we show that the $\operatorname{KL}$ divergence does not blow up in finite time, see Corollary~\ref{corr:all_gamma_KL}.
\end{itemize}

Our main contributions are as follows:
\begin{itemize}
\item We study a common variant of actor-critic where the critic is updated using temporal difference (TD) learning and the policy is updated using mirror descent. Similarly to \cite{konda,jordan}, we analyse the coupled updates in the continuous-time limit, resulting in a dynamical system where the critic flow is captured by a \textit{semi}-gradient flow and the actor flow corresponds to an approximate Fisher--Rao gradient flow over the space of probability kernels.
\item By combining convex analysis over the space of probability measures with classical Euclidean convex analysis, we develop a Lyapunov-based stability framework that captures the interplay between entropy regularisation and timescale separation, and establish stability of the resulting dynamics.
\item We prove convergence of the actor-critic dynamics for entropy-regularised MDPs with infinite action spaces. 
\end{itemize}

\subsection{Notation}

Let $(E, d)$ denote a Polish space (i.e., a complete separable metric space). We always equip a Polish space with its Borel sigma-field $\mathcal{B}(E)$. Denote by $B_b(E)$ the space of bounded measurable functions $f : E \to \mathbb{R}$ endowed with the supremum norm $|f|_{B_b(E)} = \sup_{x \in E} |f(x)|$.
Denote by $\mathcal{M}(E)$ the Banach space of finite signed measures $\mu$ on $E$ endowed with the total variation norm $|\mu|_{\mathcal{M}(E)} = |\mu|(E)$, where $|\mu|$ is the total variation measure.
Recall that if $\mu = f\, d\rho$, where $\rho \in \mathcal{M}_+(E)$ is a nonnegative measure and $f \in L^1(E, \rho)$, then $|\mu|_{\mathcal{M}(E)} = |f|_{L^1(E, \rho)}$. Denote by $\mathcal{P}(E) \subset \mathcal{M}(E)$ the set of probability measures on $E$.
Let $\delta_x \in \mathcal P(E)$ denote the Dirac measure with mass at $x\in E$.
Moreover, we denote the Euclidean norm on $\mathbb{R}^{N}$ by $|\cdot|$ with inner product $\inner{\cdot}{\cdot}$. Given some $A, B \in \mathbb{R}^{N \times N}$, by $\lambda_{\min}(A)$ we denote the minimum eigenvalue of $A$ and write $A \succeq B$ if and only if $A - B$ is positive semidefinite. 
Finally, we denote by $\left| A \right|_{\mathrm{op}}$ the operator norm of $A$ induced by the Euclidean norm, $|A|_{\mathrm{op}} := \sup_{ |x| \neq 0}\frac{|Ax|}{|x|}$.

\subsection{Entropy regularised Markov Decision Processes} Consider an infinite horizon Markov Decision Process $(S,A,P,c,\gamma)$, where the state space $S$ and action space $A$ are Polish, $P\in \mathcal{P}(S | S \times A)$ is the state transition probability kernel, $c$ is a bounded cost function and $\gamma \in (0,1)$ is a discount factor.
Let $\mu \in \mathcal{P}(A)$ denote a reference probability measure and $\tau > 0$ denote a regularisation parameter.
Let $\rho \in \mathcal P(S)$ be an arbitrary initial state distribution.
To ease notation, for each $\pi \in \mathcal{P}(A|S)$ we define the kernels $P_{\pi}(ds'|s) := \int_{A} P(ds'|s,a)\pi(da|s)$ and $P^{\pi}(ds',da'|s,a) := P(ds'|s,a)\pi(da'|s')$.
Due to~\cite[Proposition 7.28]{bertsekas1996stochastic} we can construct a probability measure $\mathbb P^\pi_\rho$, expectation $\mathbb E^\pi_\rho$ and stochastic processes $(s_n)_{n\in \mathbb N_0}$, $(a_n)_{n\in \mathbb N_0}$ with the conditional transition probabilities corresponding to those given by $P^\pi$ and $\pi$ respectively and with $s_0 \sim \rho$.
Let $\mathbb{E}_{s}^{\pi} := \mathbb{E}_{\delta_s}^{\pi}$.
For $\pi \in \mathcal{P}(A|S)$ define the regularised value function as
\begin{equation}
S\ni s \mapsto V^{\pi}_{\tau}(s)= \mathbb{E}_{s}^{\pi}\left[\sum_{n=0}^\infty\gamma^n \Big(c(s_n,a_n) + \tau \operatorname{KL}(\pi(\cdot|s_n)|\mu)\Big)\right]\in \mathbb{R}\cup \{\infty\}\,,
\end{equation} where $\operatorname{KL}(\pi(\cdot|s)|\mu)$ is the Kullback-Leibler (KL) divergence of $\pi(\cdot|s)$ with respect to $\mu$, define as $\operatorname{KL}(\pi(\cdot|s)|\mu) := \int_{A} \ln \frac{d\pi}{d\mu}(a|s) \pi(da|s)$ if $\pi(\cdot|s)$ is absolutely continuous with respect to $\mu$, and infinity otherwise.

For a given initial distribution $\rho \in \mathcal{P}(S)$, the optimal value function is defined as
\begin{equation}
\label{eq:optim}
V^{*}_{\tau}(\rho) = \min_{\pi \in \mathcal{P}(A|S)}V^{\pi}_{\tau}(\rho), \quad \text{with} \,\, V^{\pi}_{\tau}(\rho) := \int_{S} V^{\pi}_{\tau}(s)\rho(ds)
\end{equation}
and we refer to $\pi^* \in \mathcal{P}(A|S)$ as the optimal policy if $V^*_{\tau}(\rho) = V^{\pi^*}_{\tau}(\rho)$. 
The Bellman Principle for entropy regularised MDPs, see Theorem~\ref{thm:dynamics_programming}, suggests that without loss of generality, it is sufficient to consider policies from the class given by Definition~\ref{def:admissible_policies} below.

\begin{definition}[Admissible Policies]
\label{def:admissible_policies}
Let $\Pi_{\mu}$ denote the class of policies for which there exists $f \in B_{b}(S \times A)$ with \[\pi(da|s) = \frac{\exp(f(s,a))}{\int_{A} \exp(f(s,a))\mu(da)}\mu(da).\]
\end{definition}

For each $\pi \in \Pi_\mu$ the value function $V^{\pi}_{\tau}$ is the unique bounded solution of the on-policy Bellman equation 
\begin{equation}
\label{eq:on_policy_bellman}
V^{\pi}_{\tau}(s)=\int_{A}\left(Q_\tau^\pi(s,a)+\tau \ln \frac{d \pi}{d \mu}(a,s)\right)\pi(da|s)\,,    
\end{equation}
see e.g.~\cite[Lemma B.2]{david_fisher}.

For each $\pi \in \Pi_\mu$, we define the state-action value function $Q^{\pi}_{\tau}\in B_b(S\times A)$ by 
\begin{equation}\label{eq:Q_func}
Q^{\pi}_{\tau}(s,a)=c(s,a)+\gamma\int_S V_{\tau}^{\pi}(s')P(ds'|s,a)\,.
\end{equation}
We see that $Q^{\pi}_{\tau} : S \times A \to \mathbb{R}$ is a fixed point of $\mathrm{T}^{\pi} : B_{b}(S\times A) \to B_{b}(S\times A)$,  defined as
\begin{equation}\label{eq:bellman_operator_Def}
\mathrm{T}^{\pi} f(s,a) = c(s,a) + \gamma\int_{S\times A} f(s',a') P^{\pi}(ds',da'|s,a) + \tau\gamma \int_{S} \operatorname{KL}(\pi(\cdot|s')|\mu) P(ds'|s,a).
\end{equation}
As one can show this operator is a contraction, we see that $Q^{\pi}_{\tau}$ is in fact the unique fixed point.

\section{Mirror-Descent and the Fisher--Rao Gradient flow}
Let the soft advantage function be defined as 
\[
A^{\pi}_{\tau}(s,a) := Q^{\pi}_{\tau}(s,a) + \tau \ln \frac{d\pi}{d\mu}(s,a) - V^{\pi}_{\tau}(s).
\]
Then for some $\lambda > 0$ and $\pi_0 \in \Pi_{\mu}$, the Policy Mirror Descent update rule reads as
\begin{align}\label{eq:MirrorDescent}
\pi^{n+1}(\cdot|s) &= \argmin_{m \in \mathcal{P}(A)}\left[ \int_{A} A^{\pi^{n}}_{\tau}(s,a)(m(da) - \pi^{n}(da|s)) + \frac{1}{\lambda}\operatorname{KL}(m|\pi^{n}(\cdot|s)).\right] 
\end{align} 
Due to~\cite[Lemma 1.4.3]{dupuis1997weak} we know that this pointwise minimum is achieved by
\begin{equation}\label{eq:pointwise}
\frac{d\pi^{n+1}}{d\pi^{n}}(a,s) = \frac{\exp\left(-\lambda A^{\pi^{n}}_{\tau}(s,a) \right)}{\int_{A} \exp\left(-\lambda A^{\pi^{n}}_{\tau}(s,a)\right)  \pi^{n}(da|s) }.
\end{equation}
From~\eqref{eq:on_policy_bellman} we note that for any $\pi \in \mathcal{P}(A|S)$, it holds that $\int_{A} A^{\pi}_{\tau}(s,a)\pi(da|s) = 0$. Hence taking the logarithm of \eqref{eq:pointwise} we have
\[\ln \frac{d\pi^{n+1}}{d\mu}(s,a) - \ln \frac{d\pi^{n}}{d\mu}(s,a) = -\lambda A^{\pi^n}_{\tau}(s,a) - \ln \int_{A} e^{-\lambda A^{\pi^n}_{\tau}(s,a)} \pi^{n}(da|s).\]Interpolating in the time variable and letting $\lambda \to 0$ we expect to retrieve the Fisher--Rao gradient flow for the policies 
\begin{equation}\label{eq:truefisherrao}
\partial_t \ln \frac{d\pi_t}{d\mu}(s,a) =-\left(A^{\pi_t}_{\tau}(s,a) - \int_{A}A^{\pi_t}_{\tau}(s,a)\pi_t(da|s) \right) = -A^{\pi_t}_{\tau}(s,a).
\end{equation}
Note that the soft advantage formally corresponds to the functional derivative of the value function with respect to the policy $\pi^{n}$ and thus \eqref{eq:truefisherrao} can be seen as a gradient flow of the value function over the space of kernels $\mathcal{P}(A|S)$ (see \cite{david_fisher} for a detailed description of the functional derivative). 

\begin{remark}\label{remark:entropy_action}
In the case where the advantage function is fully accessible for all $t \geq 0$, \cite{david_fisher}[Theorem 2.8] shows that the entropy regularisation in the value function induces an exponential convergence to the optimal policy.

More specifically, their result shows that for all $t\geq 0$ we have 
\[
0\leq V^{\pi_t}_{\tau}(\rho) - V^{\pi^*}_{\tau}(\rho) \leq \frac{\tau}{(1-\gamma)(e^{\tau t} - 1)}\left(\int_{S} \operatorname{KL}(\pi^*(\cdot|s) | \pi_0(\cdot|s) ) d_{\rho}^{\pi^*}(ds) \right)\,.
\]
If the action space has finite cardinality and $\pi_0$ is chosen to be uniform we see that $\operatorname{KL}(\pi^*(\cdot|s)|\pi_0(\cdot|s)) \leq \log |A|$ for all $s\in S$, where $|A|$ represents the cardinality of the action space. 
One can then let $\tau \to 0$ in the above estimate to formally obtain convergence rate of order $1/t$ for the unregularised problem.

In the setting of general action spaces $\operatorname{KL}(\pi^*(\cdot|s)|\pi_0(\cdot|s))$ is finite only if the density $\frac{d\pi^*}{d\pi_0}$ exists.
However, by the dynamic programming principle for the unregularised problem~\cite[Theorem 4.2.3]{hernandez-lerma1996discrete} shows that the optimal policies will have support on a mixture of Dirac distributions.
Therefore, $\operatorname{KL}(\pi^*(\cdot|s)|\pi_0(\cdot|s))$ will be finite only if $\pi_0$ is also a mixture of Dirac distributions with support which contains the support of $\pi^*(\cdot|s)$ for all $s \in S$. 
It is not realistic to assume that one can guess the initial policy $\pi_0$ which will have the above property.
However, in the entropy regularised case, Theorem~\ref{thm:dynamics_programming} tells us $\pi^*(\cdot|s)$ has full support on $A$ and so one simply has to choose $\pi_0(\cdot|s)$ to have full support on the action space $A$ for all $s \in S$.

\end{remark}

\section{Actor Critic Methods}
Given some feature mapping $\phi : S \times A \to \mathbb{R}^N$, we parametrise the state-action value function as $Q(s,a;\theta) := \inner{\theta}{\phi(s,a)}$. Moreover, we let the approximate soft advantage function be defined as
\begin{equation}\label{eq:approximate_advantage}
A(s,a;\theta) = Q(s,a;\theta) + \tau \ln \frac{d\pi}{d\mu}(s,a) - \int_{A}\left(Q(s,a;\theta) + \tau \ln \frac{d\pi}{d\mu}(s,a) \right) \pi(da|s).
\end{equation}

The Mean Squared Bellman Error (MSBE) is defined as
\begin{equation}
\mathrm{MSBE}(\theta,\pi) = \frac{1}{2}\int_{S \times A} (Q(s,a;\theta) - \mathrm{T}^{\pi}Q(s,a;\theta))^2d_{\beta}^{\pi}(da,ds)
\end{equation}
where for some fixed $\beta \in \mathcal{P}(S \times A)$, $d_{\beta}^{\pi} \in \mathcal{P}(S\times A)$ is the state-action occupancy measure defined in \eqref{eq:occupancy}. Given that $\beta \in \mathcal{P}(S \times A)$ has full support, by \eqref{eq:bellman_operator_Def} it holds that $\mathrm{MSBE}(\theta,\pi) = 0$ if and only if $Q(s,a;\theta) = Q^{\pi}_{\tau}(s,a)$ for all $s \in S$ and $a \in A$. Hence one approach to implementing the policy mirror descent updates is to calculate the optimal parameters for $Q(s,a;\theta)$ by minimising the MSBE at each policy mirror descent iteration and then update the policy using variable steps $\left\{\lambda_{n} \right\}_{n \geq 0}$.
This reads as
\begin{equation}
\label{eq:solve_fully}
\left\{
\begin{split}
& \theta^{n+1}  = \argmin_{\theta \in \mathbb{R}^{N}} \mathrm{MSBE}(\theta,\pi^n)\,,\\
& \frac{d\pi^{n+1}}{d\pi^{n}}(a,s)  = \frac{\exp\left(-\lambda_n A(s,a;\theta^{n+1}) \right)}{\int_{A} \exp\left(-\lambda_n A(s,a;\theta^{n+1})\right)  \pi^{n}(da|s) }.
\end{split}
\right.
\end{equation}

To avoid fully solving the $\argmin$ in \eqref{eq:solve_fully} for each policy update, one can update the critic using a semi-gradient descent on a different timescale to the policy update. Let $\{h_n\}_{n\geq 0}$ be the step-sizes of the critic at iteration $n\geq 0$. Let the semi-gradient $g: \mathbb{R}^N \times \mathcal{P}(A|S) \to \mathbb{R}^N$ of the MSBE with respect to $\theta$ be
\begin{equation}
\label{eq:semi_gradient_def}
g(\theta,\pi) := \int_{S \times A}(Q(s,a;\theta) - \mathrm{T}^{\pi}Q(s,a;\theta)) \phi(s,a) d_{\beta}^{\pi}(da,ds).
\end{equation}
The full $\argmin$ update in~\eqref{eq:solve_fully} is then replaced by
\begin{equation}
\theta^{n+1} = \theta^n -h_{n}g(\theta^{n},\pi^{n}),
\end{equation}
where timescale separation $\eta_n := \frac{h_n}{\lambda_n} > 1$ ensures that the critic is updated on a much faster timescale than the policy to improve the local estimation of the policy updates. 

With general action spaces, which allow the $\operatorname{KL}$ term may be unbounded, one may need to go even further and choose a scheme which does several (and possibly increasing) number of updates of the critic before doing an actor update.

In this paper we focus on a continuous-time idealisation of the above which is presented in the next section.

\section{Dynamics}

We study the stability and convergence of the two-timescale actor-critic Mirror Descent scheme in the continuous-time limit.
Let $\eta:[0,\infty) \to [1,\infty)$ be a non-decreasing function representing the timescale separation, then for some $\theta_0 \in \mathbb{R}^N$, $\pi_0 \in \Pi_{\mu}$ and $\beta \in \mathcal{P}(S \times A)$, we have the following coupled dynamics 
\begin{align}
\label{eq:dynamics_theta}
& \frac{d\theta_t}{dt} = -\eta_t g(\theta_t,\pi_t)\,,\,\,\,\theta_0 = \theta^0 \in \mathbb R^N,\\
\label{eq:fisherrao_dynamics}
& \partial_t\pi_t(da|s) = -A(s,a;\theta_t)\pi_t(da|s)\,,\,\,t\geq 0\,,\,\,\pi_0 = \pi^0 \in \Pi_\mu,
\end{align} 
where $g: \mathbb{R}^N \times \mathcal{P}(A|S)$ is the semi-gradient of the MSBE defined in \eqref{eq:semi_gradient_def}.
We refer to \eqref{eq:fisherrao_dynamics} as the approximate Fisher--Rao Gradient flow.

We perform our analysis under the following assumptions.
\begin{assumption}[$Q^{\pi}_{\tau}$-realisability]\label{as:linearmdp}
For all $\pi \in \Pi_{\mu}$ there exists $\theta_{\pi} \in \mathbb{R}^N$ such that $Q^{\pi}(s,a) = \inner{\theta_{\pi}}{\phi(s,a)}$ for all $(s,a)\in S \times A$.
\end{assumption}
A simple example of when this holds is in the tabular case, where one can choose $\phi$ to be a one-hot encoding of the state-action space. Moreover, all linear MDPs are $Q^{\pi}$-realisable. In a linear MDP there exists exists $\phi : S\times A \to \mathbb{R}^N$, $w\in \mathbb{R}^{N}$ and a sequence $\{\psi_i\}_{i=1}^{N}$ with $\psi_i \in \mathcal{M}(S)$ such that for all $(s,a)\in S \times A$,
\[
c(s,a)=\langle w,\phi(s,a)\rangle,
\qquad
P( d s'\mid s,a)=\sum_{i=1}^{N}\phi_i(s,a)\psi_i( d s').
\]In this case it holds that $(\theta_{\pi})_i = w_i + \int_{S} V^{\pi}(s')\psi_i(ds')$. Assumption \ref{as:linearmdp} can be seen as a convention to omit function approximation errors in the final convergence results. This assumption, or the presence of approximation errors in convergence results, are widely present in the actor-critic literature (\cite{cayci}, \cite{linear2}, \cite{linear3}, \cite{stochastic_approx_application_AC},\cite{qiu2021finite_approxerror1}). 

More recently, \cite{lin2025rethinkingglobalconvergencesoftmax} derives some weaker ordering conditions in the bandit case (empty state space) which guarantee the convergence of soft-max policy gradient in the tabular setting beyond realisability. However as of now it is unclear how this applies to MDPs and also fundamentally depends on the finite cardinality of the action space.

Since for all $\pi \in \Pi_\mu$ we know that $Q^\pi_\tau \in B_b(S\times A)$ we also have $Q^\pi_\tau \in L^2(S\times A; \beta)$, which is a Hilbert space.
By~\cite[Theorem 5.11]{brezis}, Assumption \ref{as:linearmdp} holds in the limit $N \to \infty$ when $\phi_i$ are the basis functions of $L^2(S\times A;\beta)$. However, analysis in such a Hilbert space becomes more involved and intricate and is the result of ongoing work. 
Combining this approach with careful truncation of the basis functions has demonstrated empirical success in \cite{ma2024skill_quadraptor, ren2023stochastic_truncation}.

\begin{assumption}\label{as:bounded_phi}
For all $(s,a)\in S \times A$ it holds that $|\phi(s,a)| \leq 1$. 
\end{assumption}
Assumption \ref{as:bounded_phi} is purely for convention and is without loss of generality in the finite-dimensional case.
\begin{assumption}\label{as:e_value}
Let $\beta \in \mathcal{P}(S \times A)$ be fixed. Then
\[
\lambda_{\beta} := \lambda_{\min}\left( \int_{S \times A} \phi(s,a)\phi(s,a)^{\top} \, \beta(ds\,da) \right) > 0.
\]
\end{assumption}
Note that unlike the analogous assumptions imposed in \cite{stochastic_approx_application_AC}, Assumption \ref{as:e_value} is independent of the policy. This property allows us to remove any dependence on the continuity of eigenvalues.

\begin{definition}
\label{def:L_loss}
For all $\pi \in \Pi_{\mu}$ and $\zeta \in \mathcal{P}(S \times A)$, the squared loss with respect to $\zeta$ is defined as
\begin{equation}
L(\theta,\pi;\zeta) = \frac{1}{2}\int_{S \times A} (\inner{\theta}{\phi(s,a)} - Q^{\pi}_\tau(s,a))^2 \zeta(da,ds)
\end{equation}where $Q^{\pi}_{\tau}$ is defined in \eqref{eq:Q_func}.
\end{definition}

A straightforward calculation given in Lemma \ref{lemma:strong_convexity} shows that due to Lemma \ref{lemma:occupancy_prop} and Assumption \ref{as:e_value}, for any $\pi \in \Pi_{\mu}$ it holds that $L(\cdot,\pi;d_{\beta}^{\pi})$ is $(1-\gamma)\lambda_{\beta}$-strongly convex. 

The following result then connects the geometry of the semi-gradient of the MSBE and the gradient of $L(\cdot,\pi;\beta)$, which can be seen as an extension of Lemma 3 of \cite{bhandari_td_learning} to the current entropy regularised setting.
\begin{lemma}\label{lemma:gradient}
Let Assumption \ref{as:linearmdp} hold. Then for all $\theta \in \mathbb{R}^N$ and $\pi \in \Pi_{\mu}$ it holds that 
\begin{equation}
-\inner{g(\theta,\pi)}{\theta-\theta_{\pi}} \leq -(1-\sqrt{\gamma})(1-\gamma)\inner{\nabla_{\theta}{L}(\theta,\pi;\beta)}{\theta-\theta_{\pi}}
\end{equation} with \[\nabla_{\theta} L(\theta,\pi;\beta) = \int_{S \times A} (\inner{\theta}{\phi(s,a)} - Q^{\pi}_{\tau}(s,a))\phi(s,a)\beta(da,ds).\]
\end{lemma} See Appendix \ref{sec:gradient_proof} for a proof.

\section{Stability}
\label{sec:stability}

In this section we analyse the stability of the coupled actor-critic flow.
Let $(\theta_t, \pi_t)_{t\geq 0}$ be given by the system~\eqref{eq:dynamics_theta}-\eqref{eq:fisherrao_dynamics}.
Under mild assumptions,  Corollary~\ref{corr:all_gamma_KL} shows that for all $s \in S$, $\operatorname{KL}(\pi_t(\cdot|s) |\mu)$ does not blow up in finite time in the sense that there in no $T > 0$ and no $s \in S$ such that $\lim_{t \nearrow T} \operatorname{KL}(\pi_t(\cdot|s) | \mu) = +\infty$. Existence of such a time $T>0$ would result in a singularity in the actor-critic dynamics.

Throughout this section, to ease notation we let 
\[
\Gamma := \lambda_{\beta}(1-\gamma)(1-\sqrt{\gamma}),\,\,\operatorname{K}_t := \sup_{s \in S} \operatorname{KL}(\pi_t(\cdot|s)|\mu),
\]
with $\lambda_{\beta} > 0$ the constant from Assumption \ref{as:e_value}.

Using Lemma \ref{lemma:occupancy_prop}, Lemma \ref{lemma:lyap_drift} then establishes the effect of the coupling and timescale separation in the actor-critic flow and its effect on the stability of the critic parameters. 

\begin{lemma}
\label{lemma:lyap_drift}
Let Assumptions \ref{as:bounded_phi} and \ref{as:e_value} hold. Then for all $t \geq 0$ it holds that
\begin{equation}
\frac{1}{2\eta_t}\frac{d}{dt}|\theta_t|^2 \leq -\frac{\Gamma}{2}\left|\theta_t\right|^2 + \frac{\tau^2\gamma^2 \mathrm{K}_t^2}{\Gamma} + \frac{ |c|_{B_b(S \times A)}^2 }{\Gamma}
\end{equation}
\end{lemma}
See Appendix \ref{sec:lyapononv_proof} for a proof.
By connecting the result from Lemma \ref{lemma:lyap_drift} with the approximate Fisher--Rao gradient flow, we are able to establish a Gr\"onwall-type inequality for the KL divergence of the policies with respect to the reference measure. 
Lemma~\ref{lemma:lyap_drift} also illustrates that the coupled actor-critic flow is a forcing-damping system, where the damping comes from the strong convexity of the loss $\theta \mapsto L(\theta,\pi_t;d_{\beta}^{\pi_t})$ and the forcing coming from the policy updates manifesting as the $\mathrm{K}_t$ term on the right-hand-side of the estimate. Here we can see that in the finite action space setting, the forcing $\mathrm{K}_t$ term is upper bounded by a constant and thus we can arrive at stability straight away. In the current setting this is not possible and we must perform analysis over $\mathcal{P}(A|S)$ on the approximate Fisher--Rao gradient flow to arrive at stability.

\begin{theorem}\label{Theorem:gronwall_KL}
Let Assumptions \ref{as:bounded_phi} and \ref{as:e_value} hold. Let $\eta_0 > \frac{\tau}{\Gamma}$. Then there exists constants \[a_1=a_1\left(\tau,\eta_0,\gamma,\lambda_{\beta}, |c|_{B_b(S \times A)},\left|\frac{d\pi_0}{d\mu} \right|_{B_{b}(S \times A)}\right) > 0\] and $a_2 = a_2(\tau,\eta_0,\gamma,\lambda_{\beta}) > 0$ such that for all $\gamma \in (0,1)$ and $t \geq 0$ it holds that
\begin{equation}
\mathrm{K}_t^2 \leq a_1 + a_2 \int_0^t e^{-\tau(t-r)}\mathrm{K}_r^2 \,dr.
\end{equation}
\end{theorem}
See Appendix \ref{sec:gronwall_KL_sec} for a proof. Through applications of Gr\"onwall's Lemma (Lemma \ref{lemma:gronwall}), two direct corollaries of Theorem \ref{Theorem:gronwall_KL} show that the KL divergence of the policies with respect to the reference measure and the critic parameters do not blow up in finite time.
\begin{corollary}[Stability of $\pi_t$]
\label{corr:all_gamma_KL}
Under the same assumptions as Theorem \ref{Theorem:gronwall_KL}, for all $\gamma \in (0,1)$, $s \in S$ and $t \geq 0$ it holds that
\begin{equation}
\operatorname{KL}(\pi_t(\cdot|s)|\mu)^2 \leq a_1e^{a_2 t}.
\end{equation}
\end{corollary}

\begin{corollary}[Stability of $\theta_t$]
\label{corr:all_gamma_theta}
Under the same assumptions as Theorem \ref{Theorem:gronwall_KL}, suppose that there exists $\alpha > 0$ such that $\frac{d}{dt}\eta_t\leq \alpha\eta_t$. Then for all $\gamma \in (0,1)$ there exists $r_1,r_2 > 0$ such that for all $t \geq 0$ it holds that
\begin{equation}
|\theta_t| \leq r_1e^{r_2 t}.
\end{equation}
\end{corollary}

See Appendix \ref{sec:all_gamma_KL} and \ref{sec:all_gamma_theta} for the proofs.

If the MDP has sufficiently small effective time horizon due to a sufficiently small discounting factor and thus is in a sense regularised, the KL divergence of the policies with respect to the reference measure remains uniformly bounded along the flow, see
Corollaries \ref{corr:small_gamma_KL} and \ref{cor:small_gamma_theta}.

\section{Convergence}
In this section we will present final three key components before we get to the final convergence result for the coupled actor-critic flow.
First, we characterise the time derivative of the state-action value function along the approximate gradient flow for the policies. 
\begin{lemma}\label{lemma:Q_derivative}
For all $t\geq 0$ and $(s,a) \in S \times A$, it holds that
\begin{equation}
\frac{d}{dt}Q^{\pi_t}_{\tau}(s,a) = \frac{\gamma}{1-\gamma} \int_{S} \left( \int_{S \times A} A^{\pi_t}_{\tau}(s'',a'')\partial_t\pi_t(da''|s'')d^{\pi_t}(ds''|s')\right)P(ds'|s,a)
\end{equation}
\end{lemma} See Appendix \ref{sec:Q_function_derivative_proof} for a proof. Observe that in the exact setting, where $\partial_t \pi_t = -A^{\pi_t}_\tau$ as in~\eqref{eq:truefisherrao}, we obtain the dissipative property of $\{Q^{\pi_t}_{\tau} \}_{t \geq 0}$ along the flow
\[        \frac{d}{dt}Q^{\pi_t}_{\tau}(s,a) = \frac{-\gamma}{1-\gamma} \int_{S} \left( \int_{S \times A} A^{\pi_t}_{\tau}(s'',a'')^2d^{\pi_t}(ds''|s')\right)P(ds'|s,a) \leq 0.\]

Second, Theorem \ref{thm:convergence_1} shows that the actor-critic flow maintains the exponential convergence to the optimal policy induced by the $\tau$-regularisation up to a error term arising from not solving the critic to full accuracy.

\begin{theorem}\label{thm:convergence_1}Let $\left\{\pi_t,\theta_t\right\}_{t\geq 0}$ be the trajectories of the actor critic flow. Let Assumptions \ref{as:linearmdp} and \ref{as:bounded_phi} hold. Then for all $t > 0$ it holds that 
\begin{align}
\min_{r \in [0,t]} V^{\pi_r}_{\tau}(\rho)- V^{\pi^*}_{\tau}(\rho) &\leq \frac{\tau}{2(1-\gamma)(1-e^{-\frac{\tau}{2}t})} \Bigg(e^{-\frac{\tau}{2} t} \int_{S}\operatorname{KL}(\pi^*(\cdot|s)|\pi_0(\cdot|s))d_{\rho}^{\pi^*}(ds) \\
&\qquad+ \frac{1}{2\tau}\int_0^t e^{-\frac{\tau}{2}(t-r)} |\theta_r - \theta_{\pi_r}|^2 dr \Bigg)
\end{align}
\end{theorem}
See Appendix \ref{sec:conv1} for a proof. 
Theorem~\ref{thm:convergence_1} shows that the exponentially weighted error term determines the rate of convergence of the actor-critic dynamics.

Third, Theorem \ref{thm:main_thm} shows that this error term decays exponentially up to an integral which now depends on the rate of change of the true state-action value function and the timescale separation.

\begin{theorem}\label{thm:main_thm}Let Assumptions \ref{as:linearmdp}, \ref{as:bounded_phi} and \ref{as:e_value} hold. Let $\eta_0 > \frac{1}{\Gamma}$ and $0<\tau < 1$.
Then for all $t \geq 0$ there exists constants $b_1,b_2 > 0$ such that
\begin{align}
\int_0^t e^{-\frac{\tau}{2}(t-r)} |\theta_r - \theta_{\pi_r}|^2 dr \leq b_1e^{-\frac{\tau}{2}t}+ b_2\int_0^t e^{-\frac{\tau}{2}(t-r)}\frac{1}{\eta_r}\left|\frac{d}{dr}\theta_{\pi_r} \right|^2 dr.
\end{align}
\end{theorem}See Appendix \ref{sec:main_thm_proof} for a proof. 

Finally we are ready to present the main result of the paper.
Using Corollary~\ref{corr:all_gamma_KL} and by choosing $\eta_t$ such that the critic flows runs much faster than the actor, Theorem \ref{thm:all_gamma_conv} below demonstrates an exponential convergence to the optimal policy for all $\gamma \in (0,1)$.
\begin{theorem}\label{thm:all_gamma_conv}
Under the same assumptions as Theorem \ref{thm:main_thm}, there exists $k_1 > 0$ with $\eta_t = \eta_0 e^{k_1t}$ and $k_2 > 0$ such that for all $\gamma \in (0,1)$ and $t> 0$ it holds that 
\begin{align}
\min_{r \in [0,t]} V^{\pi_r}_{\tau}(\rho)- V^{\pi^*}_{\tau}(\rho) &\leq \frac{\tau e^{-\frac{\tau}{2} t}}{2(1-\gamma)(1-e^{-\frac{\tau}{2}t})} \Bigg( \int_{S}\operatorname{KL}(\pi^*(\cdot|s)|\pi_0(\cdot|s))d_{\rho}^{\pi^*}(ds) + \frac{k_2}{2\tau} \Bigg)
\end{align}
\end{theorem} See Appendix \ref{sec:proof_of_all_gamma_conv} for a proof. 

Corollary \ref{thm:small_gamma_regime} then shows that if the MDP is sufficiently regularised through a small discounting factor, one can arrive at convergence for a much more general class of timescale separation functions $t\mapsto\eta_t$.

\section{Limitations}
In this work, we only study the continuous-time dynamics of the actor-critic algorithm. Although this formulation gives insights into the discrete counterpart, a rigorous treatment of the discrete-time setting is more realistic for practical purposes and is left for future research.

Moreover, for the purposes of analysis our critic approximation is linear while in practice non-linear neural networks are used to approximate the critic. 

Finally, our work assumes all integrals are evaluated exactly, in particular the semi-gradient~\eqref{eq:semi_gradient_def}. 
In practice these would need to be estimated from samples leading to additional Monte-Carlo errors.
To fully analyse this is left for future work.

\appendix

\section{Known properties MDPs and other useful results}
\label{sec:technical}

The state-occupancy kernel $d^{\pi} \in \mathcal{P}(S|S)$ is defined by 
\begin{equation}
\label{eq:occupancy_s}
d^{\pi}(ds'|s)=(1-\gamma)\sum_{n=0}^{\infty}\gamma^nP^n_{\pi}(ds'|s)\,,
\end{equation}
where $P^n_{\pi}$ is the $n$-times product of the kernel $P_{\pi}$ with $P^0_{\pi}(ds'|s)\coloneqq \delta_s(ds')$. Moreover, for each $\pi \in \mathcal{P}(A|S)$ and $(s,a) \in S \times A$, we define the state-action occupancy kernel as
\begin{equation}
d^{\pi}(ds,da|s,a) = (1-\gamma) \sum_{n=0}^{\infty} \gamma^n (P^{\pi})^n(ds,da|s,a)
\end{equation} where $(P^{\pi})^{n}$ is the $n$-times product of the kernel $P^{\pi}$ with $(P^{\pi})^{0}(ds',da'|s,a) := \delta_{(s,a)}(ds',da')$. Given some initial state-action distribution $\beta \in \mathcal{P}(S\times A)$ with initial state distribution given by $\rho(ds) = \int_{A} \beta(da,ds)$, we define the state-occupancy and state-action occupancy measures as
\begin{equation}\label{eq:occupancy}
d^{\pi}_{\rho}(ds) = \int_{S} d^{\pi}(ds|s')\rho(ds'), \quad d^{\pi}_{\beta}(ds,da) = \int_{S \times A}d^{\pi}(ds,da|s',a') \beta(da',ds').
\end{equation}
Note that for all $E \in \mathcal{B}(S \times A)$, by defining the linear operator $J_{\pi} : \mathcal{P}(S \times A) \to  \mathcal{P}(S \times A)$ as
\begin{equation}\label{eq:transition_operator}
J_{\pi}\beta(E) = \int_{S\times A} P^{\pi}(E|s',a')\beta(ds',da'),
\end{equation}it directly holds that
\begin{equation}
d^{\pi}_{\beta}(da,ds) = (1-\gamma)\sum_{n=0}^{\infty} \gamma^n J_{\pi}^n\beta(da,ds),
\end{equation}
with $J_{\pi}^{n}$ the $n$-fold product of the operator $J_{\pi}$ with $J_{\pi}^{0}=I$, the identity operator on $\mathcal{P}(S \times A)$. The following lemma establishes properties of the state-action occupancy measure defined in \eqref{eq:occupancy} and which are useful in the proofs.

\begin{lemma}\label{lemma:occupancy_prop} For all $\pi \in \mathcal{P}(A|S)$, $\beta \in \mathcal{P}(S \times A)$ and $E \in \mathcal{B}(S \times A)$ it holds that
\begin{equation}\label{eq:property1_occupancy}
d_{J^{\pi}\beta}^{\pi}(E) = J^{\pi}d_{\beta}^{\pi}(E).
\end{equation}
Moreover, for all $\gamma \in (0,1)$ we have
\begin{equation}\label{eq:property2_occupancy}
d_{\beta}^{\pi}(E)-\gamma d_{J^{\pi}\beta}^{\pi}(E) = (1-\gamma)\beta(E).
\end{equation}
\end{lemma}
\begin{proof} For any $\beta \in \mathcal{P}(S \times A)$, $\pi \in \mathcal{P}(A|S)$ and $E \in \mathcal{B}(S \times A)$, it holds that
\begin{align}
d^{\pi}_{J_{\pi}\beta}(E) &= (1-\gamma)\sum_{n=0}^{\infty} \gamma^n (J_{\pi}^{n}J_{\pi}\beta)(E) \\
&= J_{\pi}d_{\beta}^{\pi}(E)
\end{align} where we just used the associativity of the operator $J_{\pi}$. Furthermore by letting $m=n+1$ it holds that 
\begin{align}
d_{J_{\pi}\beta}^{\pi}(E) &= (1-\gamma) \sum_{n=0}^{\infty} \gamma^n J_{\pi}^{n+1}\beta(E) \\
&= (1-\gamma)\sum_{m=1}^{\infty} \gamma^{m-1} J_{\pi}^{m}\beta(E) \\
&= \frac{1-\gamma}{\gamma}\sum_{m=1}^{\infty} \gamma^{m} J_{\pi}^{m}\beta(E) \\
&= \frac{1}{\gamma}(d_{\beta}^{\pi}(E) - (1-\gamma)\beta(E)).
\end{align}Rearranging concludes the proof.
\end{proof}

\begin{theorem}[Dynamic Programming Principle]
\label{thm:dynamics_programming}
Let $\tau > 0$. The optimal value function $V^{*}_{\tau}$ is the unique bounded solution of the following Bellman equation:
\[
V^{\ast}_{\tau}(s)=-\tau\ln\int_{A}\exp\left(-
\frac{1}{\tau}Q^{\ast}_{\tau}(s,a)\right)\mu(da),
\]
where $Q^*_{\tau}\in B_b(S\times A)$ is defined by  
\[
Q^{*}_{\tau}(s,a)=c(s,a)+\gamma\int_S V_{\tau}^{*}(s')P(ds'|s,a)\,,
\quad \forall (s,a)\in S\times A\,.
\]Moreover, there is an optimal policy $\pi^* \in \mathcal{P}(A|S)$  given by
\[
\label{eq:optimal_policy}
\pi^*(da|s) = \exp\left(-\frac{1}{\tau }(Q^{\ast}_{\tau}(s,a)-V^{\ast}_{\tau}(s))\right)\mu(da)\,,
\quad \forall s\in S.
\]Finally, the value function $V^{\pi}_{\tau}$ is the unique bounded solution of the following Bellman equation for all $s \in S$
\[
V^{\pi}_{\tau}(s)=\int_{A}\left(Q_\tau^\pi(s,a)+\tau \ln \frac{d \pi}{d \mu}(a,s)\right)\pi(da|s)\,.
\]
\end{theorem}
The performance difference lemma, first introduced for tabular unregularised MDPs, has become fundamental in the analysis of MDPs as it acts a substitute for the strong convexity of the $\pi \mapsto V^{\pi}_{\tau}$ if the state-occupancy measure $d_{\rho}^{\pi}$ is ignored (e.g \cite{kakade_PD}, \cite{jordan}, \cite{lan}). By virtue of \cite{david_fisher}, we have the following performance difference for entropy regularised MDPs in Polish state and action spaces.
\begin{lemma}[Performance difference]
\label{lem:performance_diff}
For all $\rho \in \mathcal{P}(S)$ and $\pi,\pi'\in \Pi_{\mu}$, 
\begin{align*}
&V^{\pi}_\tau(\rho)-V^{\pi'}_\tau(\rho) \\
&\quad = \frac{1}{1-\gamma}\int_S \bigg[\int_A\left(Q^{\pi'}_{\tau}(s,a)+\tau \ln \frac{d \pi'}{d\mu}(a,s)\right)(\pi-\pi')(da|s) + \tau    \operatorname{KL}(\pi(\cdot | s)|\pi'(\cdot | s)) \bigg]d^{\pi}_\rho(ds)\,.
\end{align*}
\end{lemma}

\begin{lemma}[Gr\"onwall]\label{lemma:gronwall}
Let $\lambda(s) \geq 0$, $a = a(s)$, $b = b(s)$ and $y = y(s)$ be locally integrable, real-valued functions defined on $[0,T]$ such that $y$ is also locally integrable and for almost all $s \in [0,T]$,
\[
y(s) + a(s) \leq b(s) + \int_0^s \lambda(t) y(t) dt.
\]
Then
\[
y(s) + a(s) \leq b(s) + \int_0^s \lambda(t) \left[ \int_0^t \lambda(r) (b(r) - a(r)) dr \right] dt, \quad \forall s \in [0,T].
\]

Furthermore, if $b$ is monotone increasing and $a$ is non-negative, then
\[
y(s) + a(s) \leq b(s) e^{\int_0^s \lambda(r) dr}, \quad \forall s \in [0,T].
\]
\end{lemma}

\section{Auxiliary results}

\begin{lemma}\label{lemma:main_integrals}
For some $\beta \in \mathcal{P}(S \times A)$, let $d_{\beta}^{\pi} \in \mathcal{P}(S \times A)$ be the state-action occupancy measure. Moreover let $\kappa(ds,da,ds',da') := P^{\pi}(ds',da'|s,a)d_{\beta}^{\pi}(ds,da)$. Then for any $\pi \in \Pi_{\mu}$ and any integrable $ f : S \times A \to \mathbb{R}$, it holds that 
\begin{equation}
\int_{S \times A \times S \times A} f(s,a)f(s',a') \kappa(ds,da,ds',da') \leq \frac{1}{\sqrt{\gamma}} \int_{S \times A} f(s,a)^2 d_{\beta}^{\pi}(ds,da)
\end{equation}
\end{lemma}
\begin{proof}
By Hölder's inequality, it holds that
\begin{align}\label{eq:holder}
&\int_{S \times A \times S \times A} f(s,a)f(s',a') \kappa(ds,da,ds',da')  \\
&\leq \left(\int_{S \times A \times S \times A} f(s,a)^2 \kappa(ds,da,ds',da')\right)^{\frac{1}{2}} \left(\int_{S \times A \times S \times A} f(s',a')^2 \kappa(ds,da,ds',da')\right)^{\frac{1}{2}}. 
\end{align} Moreover, observe that
\begin{align}
\int_{S \times A \times S \times A} f(s,a)^2 \kappa(ds,da,ds',da') &= \int_{S \times A} \left(\int_{S \times A} P^{\pi}(ds',da'|s,a)  \right) f(s,a)^2d_{\beta}^{\pi}(ds,da) \\
&=  \int_{S \times A} f(s,a)^2d_{\beta}^{\pi}(ds,da),
\end{align}hence \eqref{eq:holder} becomes
\begin{align}
&\left(\int_{S \times A \times S \times A} f(s,a)^2 \kappa(ds,da,ds',da')\right)^{\frac{1}{2}} \left(\int_{S \times A \times S \times A} f(s',a')^2 \kappa(ds,da,ds',da')\right)^{\frac{1}{2}}\\
&\leq \left(  \int_{S \times A} f(s,a)^2d_{\beta}^{\pi}(ds,da)\right)^{\frac{1}{2}} \left(\int_{S \times A \times S \times A} f(s',a')^2 \kappa(ds,da,ds',da')\right)^{\frac{1}{2}}.
\end{align}Now by the first part of Lemma \ref{lemma:occupancy_prop}, it holds that 
\begin{align}
\int_{S \times A \times S \times A} f(s',a')^2 \kappa(ds,da,ds',da') &= \int_{S \times A \times S \times A} f(s',a')^2P^{\pi}(ds',da'|s,a)d_{\beta}^{\pi}(ds,da) \\
&= \int_{S \times A} f(s,a)^2 d_{J^{\pi}\beta}^{\pi}(ds,da),
\end{align}where $J^{\pi} : \mathcal{P}(S \times A) \to \mathcal{P}(S \times A)$ is defined in \eqref{eq:transition_operator}. Then by the second part of Lemma \ref{lemma:occupancy_prop} we have
\begin{align}
&\left(  \int_{S \times A} f(s,a)^2d_{\beta}^{\pi}(ds,da)\right)^{\frac{1}{2}} \left(\int_{S \times A \times S \times A} f(s',a')^2 \kappa(ds,da,ds',da')\right)^{\frac{1}{2}} \\
&\leq \left(  \int_{S \times A} f(s,a)^2d_{\beta}^{\pi}(ds,da)\right)^{\frac{1}{2}}\left(\int_{S \times A} f(s,a)^2 d_{J^{\pi}\beta}^{\pi}(ds,da) \right)^{\frac{1}{2}} \\
&\leq \frac{1}{\sqrt{\gamma}} \int_{S \times A} f(s,a)^2 d_{\beta}^{\pi}(ds,da),
\end{align}which concludes the proof.
\end{proof}

To alleviate notation let $Q_t(s,a) := Q(s,a;\theta_t)$ and $A_t(s,a) := A(s,a;\theta_t)$.

\begin{lemma}\label{lemma:bounds_along_flow}
For some $\theta_0 \in \mathbb{R}^N$ and $\pi_0 \in \Pi_{\mu}$, let $\{\pi_t,\theta_t\}_{t\geq 0}$ be the trajectory of coupled actor-critic flow. Moreover let $\mathrm{K}_t = \sup_{s \in S} \operatorname{KL}(\pi_t(\cdot|s)|\mu)$. There exists $C_1 > 0 $ such that for all $t \geq 0$ it holds that
\begin{align}
&\sup_{s \in S} \left| \partial_t \pi_t(\cdot|s)\right|_{\mathcal{M}(A)} \leq \left| A_t \right|_{B_b(S\times A)}, \\
&\left| A_t\right|_{B_b(S\times A)} \leq 2\left| Q_t\right|_{B_b(S\times A)} + 2 \tau\left| \ln \frac{d\pi_t}{d\mu}\right|_{B_b(S\times A)},\\  
&\left| Q^{\pi_t}_{\tau} \right|_{B_b(S\times A)}\leq \frac{1}{1-\gamma}\left( \left|c \right|_{B_b(S\times A)} + \tau\gamma\mathrm{K}_t \right), \\
&\left| \ln \frac{d\pi_t}{d\mu}\right|_{B_b(S\times A)} \leq C_1 + \frac{2}{\tau} \sup_{r \in [0,t]} |\theta_r| + \sup_{r \in [0,t]}\mathrm{K}_r.
\end{align}
\end{lemma}
\begin{proof} 
The first claim $\sup_{s \in S} \left| \partial_t \pi_t(\cdot|s)\right|_{\mathcal{M}(A)} \leq \left| A^{\pi_t}_{\tau} \right|_{B_b(S\times A)}$ follows trivially from the definition of the approximate Fisher--Rao gradient flow defined in \eqref{eq:fisherrao_dynamics}. Moreover, it holds that
\begin{align}
\left| A_t \right|_{B_b(S\times A)} &= \left| Q_t + \tau \ln \frac{d\pi_t}{d\mu} - \int_{A} \left(Q_t(\cdot,a) + \tau \ln \frac{d\pi_t}{d\mu}(\cdot,a) \right) \pi_t(da|\cdot) \right|_{B_b(S\times A)} \\
&\leq 2\left|Q_t + \tau \ln \frac{d\pi_t}{d\mu} \right|_{B_b(S\times A)}  \\
&\leq 2\left|Q_t \right|_{B_b(S\times A)} + 2\tau \left|\ln \frac{d\pi_t}{d\mu} \right|_{B_b(S\times A)} 
\end{align}where we used the triangle inequality in the final inequality. Moreover, the state-action value function $Q^{\pi_t}_{\tau}$ is a fixed point of the Bellman operator defined in \eqref{eq:bellman_operator_Def}. Hence, for all $(s,a) \in S \times A$, we have
\begin{align}
Q^{\pi_t}_{\tau}(s,a) 
&= c(s,a) 
+ \gamma \int_{S \times A} Q^{\pi_t}_{\tau}(s',a') \, P^{\pi_t}(ds',da'|s,a) 
+ \tau \gamma \int_{S} \operatorname{KL}(\pi_t(\cdot|s') \| \mu) \, P(ds'|s,a).
\end{align}
Taking absolute values and using the triangle inequality we have
\begin{align}
\left| Q^{\pi_t}_{\tau}(s,a) \right|
&\leq \left| c \right|_{B_b(S\times A)}
+ \gamma \left| Q^{\pi_t}_{\tau} \right|_{B_b(S\times A)}
+ \tau\gamma \sup_{s' \in S} \operatorname{KL}(\pi_t(\cdot|s') \| \mu) \\
&= \left| c \right|_{B_b(S\times A)} + \gamma \left| Q^{\pi_t}_{\tau} \right|_{B_b(S\times A)} + \tau\gamma \mathrm{K}_t.
\end{align}
Taking the supremum over $(s,a) \in S\times A$ on the left-hand side yields
\begin{equation}
\left| Q^{\pi_t}_{\tau} \right|_{B_b(S\times A)} 
\leq \left| c \right|_{B_b(S\times A)} + \gamma \left| Q^{\pi_t}_{\tau} \right|_{B_b(S\times A)} + \tau\gamma \mathrm{K}_t.
\end{equation}
Rearranging gives
\begin{equation}
(1-\gamma)\left| Q^{\pi_t}_{\tau} \right|_{B_b(S\times A)} 
\leq \left| c \right|_{B_b(S\times A)} + \tau\gamma \mathrm{K}_t,
\end{equation}
which is the desired bound. Recall the approximate Fisher--Rao gradient flow for the policies $\{\pi_t\}_{t \geq 0}$, which for all $t \geq 0$ and for all $(s,a) \in S \times A$ is given by
\begin{equation}
\partial_t \ln \frac{d\pi_t}{d\mu}(s,a) 
= -\left( Q_t(s,a) + \tau \ln\frac{d\pi_t}{d\mu}(a,s) 
- \int_{A} \left(Q_t(s,a') + \tau \ln\frac{d\pi_t}{d\mu}(a',s)\right)\pi_t(da'|s)\right).
\end{equation}
Duhamel's principle yields for all $t \geq 0$ that
\begin{align}\label{eq:duhamel}
\ln \frac{d\pi_t}{d\mu}(s,a) 
&= e^{-\tau t}\ln \frac{d\pi_0}{d\mu}(a,s) 
+ \int_0^t e^{-\tau(t-r)} \left( \int_{A} Q_r(s,a')\pi_r(da'|s) - Q_r(s,a) \right) dr \\
&\quad + \tau\int_0^t e^{-\tau(t-r)} \operatorname{KL}(\pi_{r}(\cdot|s) | \mu) dr.
\end{align}
Since $\pi_0 \in \Pi_{\mu}$, there exists $C_1 \geq 1$ such that $\left| \ln \frac{d\pi_0}{d \mu} \right|_{B_b(S \times A)} \leq C_1$.
Then by Assumption~\ref{as:bounded_phi} we have that for all $t \geq 0$,
\begin{align}
\left| \ln \frac{d\pi_t}{d\mu}(s,a) \right| 
&\leq C_1 
+ \int_0^t e^{-\tau(t-r)} \left| \int_{A} Q_r(s,a')\pi_r(da'|s) - Q_r(s,a) \right| dr \\
&\quad + \tau \int_0^t e^{-\tau(t-r)} \operatorname{KL}(\pi_r(\cdot|s) \| \mu) \, dr \\
&\leq C_1 
+ 2\int_0^t e^{-\tau(t-r)} \left| \theta_r \right| dr 
+ \tau \int_0^t e^{-\tau(t-r)} \mathrm{K}_r \, dr \\
&\leq C_1 + \frac{2}{\tau} \sup_{r \in [0,t]} \left| \theta_r \right| 
+ \sup_{r \in [0,t]} \mathrm{K}_r,
\end{align}
where in the last inequality we used $\int_0^t e^{-\tau(t-r)} dr \leq \frac{1}{\tau}$.  
Taking the supremum over $(s,a) \in S \times A$ yields
\begin{equation}
\left| \ln \frac{d\pi_t}{d\mu} \right|_{B_b(S \times A)} 
\leq C_1 + \frac{2}{\tau} \sup_{r \in [0,t]} \left| \theta_r \right| 
+ \sup_{r \in [0,t]} \mathrm{K}_r,
\end{equation}
which is the desired bound.
\end{proof}

\begin{lemma}\label{lemma:strong_convexity}
Let Assumption \ref{as:e_value} hold. Then for all $\pi \in \Pi_{\mu}$, it holds that $L(\cdot,\pi;d_{\beta}^{\pi})$ is $\lambda_{\beta}(1-\gamma)$-strongly convex.
\end{lemma}
\begin{proof}
For any $\xi \in \mathcal{P}(S \times A)$, let $\Sigma_{\xi} := \int_{S \times A} \phi(s,a)\phi(s,a)^\top \xi(ds,da) \in \mathbb{R}^{N \times N}$. Then by Lemma \ref{lemma:occupancy_prop} and Assumption \ref{as:e_value} it holds that $\Sigma_{d_{\beta}^{\pi}
} \succeq (1-\gamma)\Sigma_{\beta} \succeq (1-\gamma)\lambda_{\beta} I$ and thus $L(\cdot,\pi;d_{\beta}^{\pi})$ is $\lambda_{\beta}(1-\gamma)$-strongly convex.
\end{proof}

\subsection{Proof of Lemma \ref{lemma:gradient}}\label{sec:gradient_proof}
\begin{proof}
Recall that $Q(s,a) = \inner{\theta}{\phi(s,a)}$ for some $\theta \in \mathbb{R}^N$ and that for all $\pi \in \Pi_{\mu}$, there exists $\theta_{\pi} \in \mathbb{R}^{N}$ such that $Q^{\pi}(s,a) = \inner{\theta_{\pi}}{\phi(s,a)}$ by Assumption \ref{as:linearmdp}. Then by definition of the semi-gradient of the MSBE $g : \mathbb{R}^N \times \mathcal{P}(A|S) \to \mathbb{R}^{N}$ in \eqref{eq:semi_gradient_def}, it holds that 
\begin{align}
&\inner{g(\theta,\pi)}{\theta - \theta_{\pi}} = \inner{\int_{S \times A}\left(Q(s,a) - \mathrm{T}^{\pi}Q(s,a)\right)\phi(s,a)d_{\beta}^{\pi}(da,ds) }{\theta - \theta_{\pi}} \\
&= \inner{\int_{S \times A}(Q(s,a) - Q^{\pi}_{\tau}(s,a))\phi(s,a) d_{\beta}^{\pi}(da,ds)}{\theta - \theta_{\pi}} \\
&\qquad+ \inner{\int_{S \times A}(Q^{\pi}_{\tau}(s,a) - \mathrm{T}^{\pi}Q(s,a)\phi(s,a) d_{\beta}^{\pi}(da,ds)}{\theta - \theta_{\pi}} \\
&= \inner{\int_{S \times A}(Q(s,a) - Q^{\pi}_{\tau}(s,a))\phi(s,a) d_{\beta}^{\pi}(da,ds)}{\theta - \theta_{\pi}} \\
&\qquad-\gamma \inner{\int_{S \times A \times S \times A}(Q(s',a') - Q^{\pi}_{\tau}(s',a'))\phi(s,a)P^{\pi}(ds',da'|s,a)d_{\beta}^{\pi}(ds,da)}{\theta - \theta_{\pi}}, 
\end{align} 
where we added and subtracted the true state-action value function $Q^{\pi}_{\tau} \in B_{b}(S \times A)$ in the second equality and used the fact that it is a fixed point of the Bellman operator defined in \eqref{eq:bellman_operator_Def}. To ease notation, let $\varepsilon(s,a) := Q(s,a) - Q^{\pi}_{\tau}(s,a)$. Multiplying both sides by $-1$ and using the associativity of the inner product, we have
\begin{align}
&-\inner{g(\theta,\pi)}{\theta - \theta_{\pi}} \\
&= -\inner{\int_{S \times A}\varepsilon(s,a)\phi(s,a) d_{\beta}^{\pi}(da,ds)}{\theta - \theta_{\pi}} \\
&\qquad+\gamma \inner{\int_{S \times A}\varepsilon(s',a')\phi(s,a)P^{\pi}(ds',da'|s,a)d_{\beta}^{\pi}(ds,da)}{\theta - \theta_{\pi}} \\
&=-\int_{S \times A}\varepsilon(s,a)\inner{\phi(s,a)}{\theta - \theta_{\pi}} d_{\beta}^{\pi}(da,ds) \\
&\qquad+\gamma \int_{S \times A}\varepsilon(s',a')\inner{\phi(s,a)}{\theta - \theta_{\pi}}P^{\pi}(ds',da'|s,a)d_{\beta}^{\pi}(ds,da)\\
\label{eq:sub_into_here}
&=-\int_{S \times A} \varepsilon(s,a)^2 d_{\beta}^{\pi}(da,ds) \\
\label{eq:CS_2}
&\qquad+ \gamma\int_{S \times A \times S \times A} \varepsilon(s,a)\varepsilon(s',a')P^{\pi}(ds',da'|s,a)d_{\beta}^{\pi}(ds,da) \\
&= I^{(1)} + \gamma I^{(2)}.
\end{align}Now applying Lemma \ref{lemma:main_integrals} to $I^{(2)}$ we have
\begin{align}
I^{(2)} &:= \int_{S \times A \times S \times A} \varepsilon(s,a)\varepsilon(s',a')P^{\pi}(ds',da'|s,a)d_{\beta}^{\pi}(ds,da) \\ 
&\leq \frac{1}{\sqrt{\gamma}} \int_{S \times A} \varepsilon(s,a)^2 d_{\beta}^{\pi}(ds,da).
\end{align}Thus it holds that
\begin{align}
-\inner{g(\theta,\pi)}{\theta - \theta_{\pi}} &\leq I^{(1)} + \gamma I^{(2)} \\
&\leq-(1-\sqrt{\gamma})\int_{S \times A} \epsilon(s,a)^2 d_{\beta}^{\pi}(da,ds) \\
&= -(1-\sqrt{\gamma}) \int_{S \times A}(Q(s,a) - Q^{\pi}_{\tau}(s,a))^2 d_{\beta}^{\pi}(da,ds)\\
&=-(1-\sqrt{\gamma})\inner{\nabla_{\theta}L(\theta,\pi;d_{\beta}^{\pi})}{\theta-\theta_{\pi}},
\end{align} where the last inequality follows from the Assumption \ref{as:linearmdp} and the definition of $Q(s,a) = \inner{\theta}{\phi(s,a)}$.

\end{proof}

\section{Proof of Stability Results}
\subsection{Proof of Lemma \ref{lemma:lyap_drift}}\label{sec:lyapononv_proof}
\begin{proof}

Consider the following equation
\begin{align}\label{eq:originallyaponov}
\frac{1}{2\eta_t}\frac{d}{dt} \left | \theta_t \right|^2 
&= \frac{1}{\eta_t} \inner{\frac{d}{dt}\theta_t}{\theta_t} \\
&= - \inner{g(\theta_t,\pi_t)}{\theta_t} \\
&=  -\inner{\int_{S\times A} \left( Q_t(s,a) - T^{\pi_t}Q_t(s,a) \right) \phi(s,a)\, d_{\beta}^{\pi_t}(da,ds)}{\theta_t} \\
\label{eq:thm1_Qphi}
&= -\inner{\int_{S\times A} Q_t(s,a)\phi(s,a)\, d_{\beta}^{\pi_t}(da,ds)}{\theta_t} \\
\label{eq:thm1_TQphi}
&\qquad + \inner{\int_{S \times A} T^{\pi_t}Q_t(s,a)\phi(s,a)\, d_{\beta}^{\pi_t}(da,ds)}{\theta_t} \\
&:= -J^{(1)}_t + J^{(2)}_t
\end{align}where we used the $\theta_t$ dynamics from \eqref{eq:dynamics_theta} in the second equality and the definition of the semi-gradient in the third equality. For any $\pi \in \Pi_{\mu}$, let $ \Sigma^{\pi} \in \mathbb{R}^{N \times N}$ be
\begin{equation}
\Sigma^{\pi} = \int_{S \times A} \phi(s,a)\phi(s,a)^{\top} d_{\beta}^{\pi}(da,ds).
\end{equation} Then by definition we have that $Q_t(s,a) = \inner{\theta_t}{\phi(s,a)}$, hence for $J^{(1)}_t$ we have
\begin{align}
J^{(1)}_t &=\inner{\int_{S\times A} Q_t(s,a)\phi(s,a)\, d_{\beta}^{\pi_t}(da,ds)}{\theta_t} \\
&= \inner{\int_{S\times A}\inner{\theta_t}{\phi(s,a)}\phi(s,a)d_{\beta}^{\pi_t}(da,ds)}{\theta_t} \\
&= \inner{\theta_t}{\left(\int_{S \times A} \phi(s,a) \phi(s,a)^{\top}d_{\beta}^{\pi_t}(da,ds)\right)\theta_t} \\
\label{eq:PSD_eq}
&= \inner{\theta_t}{\Sigma^{\pi_t}\theta_t}
\end{align} 
Now dealing with $J^{(1)}_t$, expanding the Bellman operator defined in \eqref{eq:bellman_operator_Def} we have
\begin{align} \label{eq:one_before_CS}
J^{(2)}_t &= \inner{\int_{S \times A} \mathrm{T}^{\pi_t}Q_t(s,a)\phi(s,a)\, d_{\beta}^{\pi_t}(da,ds)}{\theta_t} \\
&= \inner{\int_{S \times A}c(s,a) \phi(s,a) d_{\beta}^{\pi_t}(da,ds)}{\theta_t} \\
\label{eq:CS_integral}
&\qquad+ \gamma \inner{\int_{S \times A} \inner{\theta_t}{\phi(s',a')}\phi(s,a)P^{\pi_t}(ds',da'|s,a)d_{\beta}^{\pi_t}(da,ds)}{\theta_t} \\
&\qquad+ \tau\gamma \inner{\int_{S \times A} \left(\int_{S} \operatorname{KL}(\pi_t(\cdot|s'),\mu)P(ds'|s,a)\phi(s,a)d_{\beta}^{\pi_t}(da,ds)\right)}{\theta_t} \\
&\leq |c|_{B_b{(S\times A)}}|\theta_t| + \gamma I^{(1)}_t + \tau\gamma I^{(2)}_t
\end{align}where we defined \[I_t^{(1)} = \inner{\int_{S \times A} \inner{\theta_t}{\phi(s',a')}\phi(s,a)P^{\pi_t}(ds',da'|s,a)d_{\beta}^{\pi_t}(da,ds)}{\theta_t}, \] \[I^{(2)}_t = \inner{\int_{S \times A} \left(\int_{S} \operatorname{KL}(\pi_t(\cdot|s'),\mu)P(ds'|s,a)\phi(s,a)d_{\beta}^{\pi}(da,ds)\right)}{\theta_t}.\]
Moreover, to ease notation let \[\mathrm{\mathrm{K}_t} := \sup_{s \in S} \operatorname{KL}(\pi_t(\cdot|s)|\mu)\] and temporarily let $\kappa_t(ds,da,ds',da') := P^{\pi_t}(ds',da'|s,a)d_{\beta}^{\pi_t}(da,ds)$. Now focusing on $I_t^{(1)}$, it holds that
\begin{align}
I^{(1)}_t&=\inner{\int_{S \times A \times S \times A} \inner{\theta_t}{\phi(s',a')}\phi(s,a)\kappa_t(da',ds',da,ds)}{\theta_t} \\
&= \int_{S \times A \times S \times A}\inner{\theta_t}{\phi(s,a)}\inner{\theta_t}{\phi(s',a')} \kappa_t(ds',da',ds,da).
\end{align}Now using Lemma \ref{lemma:main_integrals} with $f = \inner{\theta}{\phi(\cdot,\cdot)}$ we have
\begin{align}
I^{(1)}_t &\leq
\frac{1}{\sqrt{\gamma}} \left( \int_{S \times A } \inner{\theta_t}{\phi(s,a)}^2d_{\beta}^{\pi_t}(ds,da) \right)^{\frac{1}{2}}\left( \int_{S \times A } \inner{\theta_t}{\phi(s,a)}^2d_{\beta}^{\pi_t}(ds,da) \right)^{\frac{1}{2}} \\
&= \frac{1}{\sqrt{\gamma}} \int_{S \times A } \inner{\theta_t}{\phi(s,a)}^2d_{\beta}^{\pi_t}(ds,da) \\
&= \frac{1}{\sqrt{\gamma}}\inner{\theta_t}{\Sigma^{\pi_t}\theta_t}.
\end{align}Thus all together it holds that
\begin{equation}
\gamma I_t^{(1)} \leq \sqrt{\gamma} \inner{\theta_t}{\Sigma^{\pi_t}\theta_t}.
\end{equation}
Now focusing on $I_t^{(2)}$, we have
\begin{align}
I_t^{(2)} &= \inner{\int_{S \times A} \left(\int_{S} \operatorname{KL}(\pi_t(\cdot|s'),\mu)P(ds'|s,a)\right)\phi(s,a)d_{\beta}^{\pi_t}(da,ds)}{\theta_t} \\
&\leq \mathrm{\mathrm{K}_t} \left|\int_{S \times A} \phi(s,a)d_{\beta}^{\pi_t}(ds,da)  \right| |\theta_t|\\
&\leq \mathrm{\mathrm{K}_t} |\theta_t|
\end{align}
where we used Assumption \ref{as:bounded_phi} in the final inequality. Hence along with \eqref{eq:PSD_eq}, \eqref{eq:originallyaponov} becomes
\begin{align}\label{eq:updated_lyap}
\frac{1}{2\eta_t}\frac{d}{dt}|\theta_t|^2  &\leq -J^{(1)}_t + J^{(2)}_t \\
&\leq -\inner{\theta_t}{\Sigma^{\pi_t}\theta_t} + |c|_{B_b{(S\times A)}}|\theta_t| + \gamma I^{(1)}_t + \tau\gamma I^{(2)}_t \\
&\leq -\inner{\theta_t}{\Sigma^{\pi_t}\theta_t} + \sqrt{\gamma}\inner{\theta_t}{\Sigma^{\pi_t}\theta_t} + |c|_{B_b(S\times A)}|\theta_t| + \tau\gamma\mathrm{K}_t|\theta_t| \\
&= -(1-\sqrt{\gamma})\inner{\theta_t}{\Sigma^{\pi_t}\theta_t} + \left(|c|_{B_b(S\times A)} + \tau\gamma\mathrm{K}_t \right)|\theta_t|.
\end{align} Observe that by \eqref{eq:property2_occupancy} and Assumption \ref{as:bounded_phi}, $\Sigma^{\pi} \in \mathbb{R}^{N \times N}$ is positive definite for all $\pi \in \mathcal{P}(A|S)$, hence it holds that
\begin{equation}
\inner{\theta_t}{\Sigma^{\pi_t}\theta_t} \geq (1-\gamma)\lambda_{\beta}\left|\theta_t\right|^2.
\end{equation}Therefore \eqref{eq:updated_lyap} becomes 
\begin{equation}
\frac{1}{2\eta_t}\frac{d}{dt}|\theta_t|^2 \leq -(1-\sqrt{\gamma})(1-\gamma)\lambda_{\beta}\left|\theta_t\right|^2 + (|c|_{B_b(S \times A)} + \tau\gamma \mathrm{K}_t)|\theta_t|
\end{equation}Let $\Gamma := \lambda_{\beta}(1-\gamma)(1-\sqrt{\gamma})$. By Young's inequality, there exists $\epsilon > 0$ such that
\begin{align}
\frac{1}{2\eta_t}\frac{d}{dt}|\theta_t|^2 &\leq -\Gamma|\theta_t|^2 + \frac{\epsilon}{2}|\theta_t|^2 + \frac{(|c|_{B_b(S \times A)} + \tau\gamma\mathrm{K}_t)^2}{2\epsilon} \\
&\leq -\Gamma|\theta_t|^2 + \frac{\epsilon}{2}|\theta_t|^2 + \frac{|c|_{B_b(S \times A)}^2 + \tau^2\gamma^2\mathrm{K}_t^2}{\epsilon},
\end{align}where we used the identity $(a+b)^2 \leq 2a^2 + 2b^2$. Choosing $\epsilon = \Gamma$ we arrive at
\begin{equation}
\frac{1}{2\eta_t}\frac{d}{dt}|\theta_t|^2 \leq -\frac{\Gamma}{2}\left|\theta_t\right|^2 + \frac{\tau^2\gamma^2 \mathrm{K}_t^2}{\Gamma} + \frac{ |c|_{B_b(S \times A)}^2 }{\Gamma}
\end{equation} which concludes the proof.
\end{proof}

\subsection{Proof of Theorem \ref{Theorem:gronwall_KL}}
\label{sec:gronwall_KL_sec}
\begin{proof}
By Lemma \ref{lemma:lyap_drift}, we have that for all $r \geq 0$
\begin{equation}
\frac{1}{2\eta_r}\frac{d}{dr}|\theta_r|^2 \leq -\frac{\Gamma}{2}\left|\theta_r\right|^2 + \frac{\tau^2\gamma^2 \mathrm{K}_r^2}{\Gamma} + \frac{ |c|_{B_b(S \times A)}^2 }{\Gamma}.
\end{equation}Rearranging, it holds that for all $t \geq 0$
\begin{align}
|\theta_r|^2 \leq -\frac{1}{\Gamma\eta_r}\frac{d}{dr}|\theta_r|^2  + \frac{2|c|_{B_b(S \times A)}^2 + 2\tau^2\gamma^2\mathrm{K}_r^2}{\Gamma^2}.
\end{align} Multiplying both sides by $e^{-\tau(t-r)}$ and integrating over $r$ from $0$ to $t$ we have that for all $t \geq 0$
\begin{align}\label{eq:gronwall_Step1}
\int_0^t e^{-\tau(t-r)}|\theta_r|^2 dr &\leq -\frac{1}{\Gamma}\int_0^t e^{-\tau(t-r)}\frac{1}{\eta_r}\frac{d}{dr}|\theta_r|^2  dr + \frac{2|c|_{B_b(S \times A)}^2}{\Gamma^2}\int_0^t e^{-\tau(t-r)} dr \\
&\qquad+ \frac{2\tau^2 \gamma^2}{\Gamma^2}\int_0^t e^{-\tau(t-r)}\mathrm{K}_r^2 dr \\
&\leq -\frac{1}{\Gamma}\int_0^t e^{-\tau(t-r)}\frac{1}{\eta_r}\frac{d}{dr}|\theta_r|^2  dr + \frac{2|c|_{B_b(S \times A)}^2}{\Gamma^2\tau} +\frac{2\tau^2 \gamma^2}{\Gamma^2}\int_0^t e^{-\tau(t-r)}\mathrm{K}_r^2 dr,
\end{align} where we used that $\int_0^t e^{-\tau(t-r)} dr \leq \frac{1}{\tau}$. Integrating the first term by parts, we have
\begin{align}\label{eq:int_by_parts}
-\int_0^t e^{-\tau(t-r)}\frac{1}{\eta_r}\frac{d}{dr}|\theta_r|^2  dr &= -\frac{|\theta_t|^2}{\eta_t} + e^{-\tau t}\frac{|\theta_0|^2}{\eta_0} + \tau\int_0^t |\theta_r|^2\frac{e^{-\tau(t-r)}}{\eta_r} dr  \\
&- \int_0^t |\theta_r|^2 \frac
{e^{-\tau(t-r)}\frac{d}{dr}\eta_r}{\eta_r^2}dr.
\end{align}Since by definition we have that for all $t\geq 0$, $\eta_t \geq 1$ and $\frac{d}{dt}\eta_t \geq 0$ it holds that 
\begin{equation}
\int_0^t |\theta_r|^2 \frac
{e^{-\tau(t-r)}\frac{d}{dr}\eta_r}{\eta_r^2}dr \geq 0.
\end{equation}Hence dropping the negative terms on the right hand side of \eqref{eq:int_by_parts} and using that $\eta_t \geq \eta_0$ for all $t \geq 0$, we have
\begin{align}
-\frac{1}{\Gamma}\int_0^t e^{-\tau(t-r)}\frac{1}{\eta_r}\frac{d}{dr}|\theta_r|^2  dr &\leq e^{-\tau t}\frac{|\theta_0|^2}{\Gamma\eta_0} + \frac{\tau}{\Gamma\eta_0}\int_0^t e^{-\tau(t-r)}|\theta_r|^2 dr.
\end{align}Substituting this back into \eqref{eq:gronwall_Step1}, for all $t \geq 0$ we have that
\begin{align}
\int_0^t e^{-\tau(t-r)}|\theta_r|^2 dr &\leq e^{-\tau t}\frac{|\theta_0|^2}{\Gamma\eta_0} + \frac{\tau}{\Gamma\eta_0}\int_0^t e^{-\tau(t-r)}|\theta_r|^2 dr \\
&+ \frac{2|c|_{B_b(S \times A)}^2}{\Gamma^2\tau} + \frac{2\tau^2 \gamma^2}{\Gamma^2}\int_0^t e^{-\tau(t-r)}\mathrm{K}_r^2 dr.
\end{align}Grouping like terms we have
\begin{align}
\left(1-\frac{\tau}{\Gamma\eta_0}\right)\int_0^t e^{-\tau(t-r)}|\theta_r|^2 dr \leq e^{-\tau t}\frac{|\theta_0|^2}{\Gamma\eta_0}+ \frac{2|c|_{B_b(S \times A)}^2}{\Gamma^2\tau} + \frac{2\tau^2 \gamma^2}{\Gamma^2}\int_0^t e^{-\tau(t-r)}\mathrm{K}_r^2 dr.
\end{align} Recall that we have $\eta_0 > \frac{\tau}{\Gamma}$ to ensure that $1-\frac{\tau}{\Gamma\eta_0} > 0$. Dividing through by $1-\frac{\tau}{\Gamma\eta_0}$ gives for all $t \geq 0$ that
\begin{align}\label{eq:thetaweighted}
\int_0^t e^{-\tau(t-r)}|\theta_r|^2 dr \leq  \sigma_1 + \sigma_2 \int_0^t e^{-\tau(t-r)}\mathrm{K}_r^2 dr
\end{align} where we've set \[\sigma_1 : = \frac{|\theta_0|^2}{\Gamma \eta_0 \left(1-\frac{\tau}{\Gamma\eta_0} \right)} + \frac{2 |c|_{B_b(S \times A)}^2}{\Gamma^2 \tau \left( 1-\frac{\tau}{\Gamma\eta_0}\right)},\] \[\sigma_2 := \frac{2\tau^2 \gamma^2}{\Gamma^2 \left( 1- \frac{\tau}{\Gamma\eta_0}\right)}.\] Recall the approximate Fisher--Rao gradient flow for the policies $\left\{\pi_t\right\}_{t\geq 0}$, which for all $t \geq 0$ and for all $s \in S$, $a \in A$ is
\begin{equation}
\partial_t \ln \frac{d\pi_t}{d\mu}(s,a) = -\left( Q_t(s,a) + \tau \ln\frac{d\pi_t}{d\mu}(a,s) - \int_{A} \left(Q_t(s,a) + \tau \ln\frac{d\pi_t}{d\mu}(a,s)\right)\pi_t(da|s)\right)
\end{equation} Duhamel's principle yields for all $t\geq 0$ that
\begin{align}\label{eq:duhamel}
\ln \frac{d\pi_t}{d\mu}(s,a) = e^{-\tau t}\ln \frac{d\pi_0}{d\mu}(a,s) &+ \int_0^t e^{-\tau(t-r)} \left(\int_{A} Q_r(s,a)\pi_r(da|s)
- Q_r(s,a) \right) dr \\
&+ \tau\int_0^t e^{-\tau(t-r)} \operatorname{KL}(\pi_{r}(\cdot|s)|\mu)dr
\end{align}
Observe that since $\pi_0 \in \Pi_{\mu}$, there exists $C_1 \geq 1$ such that $ \ln \left| \frac{d\pi_t}{d \mu} \right|_{B_b(S \times A)} \leq C_1$. Using that $e^{-\tau t} \leq 1$ and assumption \ref{as:bounded_phi} gives that for all $t\geq 0$ 
\begin{align}
\ln \frac{d\pi_t}{d\mu}(s,a) &\leq  C_1 + 2\int_0^t e^{-\tau(t-r)}|\theta_r| dr + \tau \int_0^t e^{-\tau(t-r)} \operatorname{KL}(\pi_r(\cdot|s)|\mu) dr \\
&\leq  C_1 +  2\int_0^t e^{-\tau(t-r)}|\theta_r| dr + \tau \int_0^t e^{-\tau(t-r)} \mathrm{K
}_r dr 
\end{align} Integrating over the actions with respect to $\pi_t(\cdot|s) \in \mathcal{P}(A)$ gives for all $t \geq 0$ that 
\begin{equation}
\operatorname{KL}(\pi_t(\cdot|s)|\mu) \leq C_1 +  2\int_0^t e^{-\tau(t-r)}|\theta_r| dr + \tau \int_0^t e^{-\tau(t-r)} \mathrm{K
}_r dr 
\end{equation} where we again use that $\mathrm{K}_r = \sup_{s \in S} \operatorname{KL}(\pi_r(\cdot|s) | \mu)$. Following from the techniques in \cite{polyak}, observe that from \eqref{eq:duhamel} and Assumption \ref{as:bounded_phi} we similarly get for all $t\geq 0$ that
\begin{equation}
\ln \frac{d\mu}{d\pi_t}(a,s) = -\ln \frac{d\pi_t}{d\mu}(s,a) \leq C_1 +  2\int_0^t e^{-\tau(t-r)}|\theta_r| dr - \tau \int_0^t e^{-\tau(t-r)} \mathrm{K
}_r dr. 
\end{equation} Now integrating over the actions with respect to the reference measure $\mu \in \mathcal{P}(A)$ we have
\begin{equation}
\operatorname{KL}(\mu | \pi_t(\cdot|s)) \leq C_1 +  2\int_0^t e^{-\tau(t-r)}|\theta_r| dr - \tau \int_0^t e^{-\tau(t-r)} \mathrm{K
}_r dr 
\end{equation}Moreover, using the non-negativity of the KL divergence, it holds for all $t \geq 0$ that
\begin{equation}
\operatorname{KL}(\pi_t(\cdot|s)|\mu) \leq \operatorname{KL}(\pi_t(\cdot|s)|\mu) + \operatorname{KL}(\mu | \pi_t(\cdot|s)) \leq 2C_1 +  4\int_0^t e^{-\tau(t-r)}|\theta_r| dr
\end{equation}
Since this holds for any $s \in S$, it holds for all $t\geq 0$ that
\begin{equation}
\mathrm{K
}_t \leq 2C_1 +  4\int_0^t e^{-\tau(t-r)}|\theta_r| dr
\end{equation} Now squaring both sides and using the H\"older's inequality, we have
\begin{align}
\mathrm{K
}_t^2 &\leq \left(2 C_1 +  4\int_0^t e^{-\tau(t-r)}|\theta_r| dr  \right)^2 \\
&\leq 8(C_1)^2 +  32\left(\int_0^t e^{-\tau(t-r)}|\theta_r| dr \right)^2 \\ 
&= 8(C_1)^2 +  32\left(\int_0^t e^{-\frac{\tau}{2}(t-r)}e^{-\frac{\tau}{2}(t-r)}|\theta_r| dr \right)^2\\
& \leq 8(C_1)^2 + 32\left(\int_0^t e^{-\tau(t-r)} dr\right) \left(\int_0^t e^{-\tau(t-r)} |\theta_r|^2 dr\right)\\\label{eq:KL_gronwall_step2}
&\leq 8(C_1)^2 +  \frac{32}{\tau}\int_0^t e^{-\tau(t-r)}|\theta_r|^2 dr,
\end{align}where we again used $\int_0^t e^{-\tau(t-r)} dr \leq \frac{1}{\tau}$. We can now substitute \eqref{eq:thetaweighted} into \eqref{eq:KL_gronwall_step2} to arrive at
\begin{align}
\mathrm{K
}_t^2 &\leq 8(C_1)^2 +  \frac{32}{\tau}\sigma_1 + \frac{32}{\tau}\sigma_2\int_0^t e^{-\tau(t-r)}\mathrm{K}_r^2 dr\\
&:= a_1 + a_2\int_0^t e^{-\tau(t-r)}\mathrm{K}_r^2 dr \label{eq:bellman_type}
\end{align} with $a_1 = 8(C_1)^2 + \frac{32}{\tau}\sigma_1$ and $a_2 = \frac{32\sigma_2}{\tau}$. 
\end{proof}

\subsection{Proof of Corollary \ref{corr:all_gamma_KL}}\label{sec:all_gamma_KL}
\begin{proof}
By Theorem \ref{Theorem:gronwall_KL} it holds that
\begin{equation}
\mathrm{K}_t^2 \leq a_1 + a_2\int_0^t e^{-\tau(t-r)}\mathrm{K}_r^2 dr.
\end{equation}
Observe that by multiplying through by $e^{\tau t}$, we can rewrite this as
\begin{equation}
e^{\tau t}\mathrm{K}_t^2 \leq e^{\tau t}a_1 + a_2\int_0^t e^{\tau r}\mathrm{K}_r^2 dr.
\end{equation}Hence after defining $g(t) = e^{\tau t} \mathrm{K}_t^2$ and applying Gr\"onwall's inequality (Lemma \ref{lemma:gronwall}), for all $\gamma \in (0,1)$ it holds for all $t \geq 0$ that
\begin{equation}
\mathrm{K}_t^2 \leq a_1e^{a_2 t}.
\end{equation}
\end{proof}

\subsection{Proof of Corollary \ref{corr:all_gamma_theta}}\label{sec:all_gamma_theta}
\begin{proof}
By Corollary \ref{corr:all_gamma_KL} and Lemma \ref{lemma:lyap_drift}, for all $\gamma \in (0,1)$ it holds that
\begin{align}
\frac{1}{2}\frac{d}{dt}\left|\theta_t \right|^2 &\leq -\frac{\Gamma}{2}\eta_t|\theta_t|^2 + b_t\eta_t
\end{align} such that 
\begin{equation}
b_t = \left(\frac{2|c|_{B_b(S \times A)}^2 + 2\tau^2\gamma^2 a_1e^{a_2t }}{\Gamma^2}\right).
\end{equation}Recall that there exists $\alpha > 0$ such that $\frac{d}{dt}\eta_t \leq \alpha{\eta_t}$, then another application of Gr\"onwall's Lemma then concludes the proof.
\end{proof}

\section{Proof of Convergence Results}
\subsection{Proof of Lemma \ref{lemma:Q_derivative}}
\label{sec:Q_function_derivative_proof}
\begin{proof}
By the definition of the state-action value function \eqref{eq:Q_func} it holds that 
\begin{align}
\frac{d}{dt}Q^{\pi_t}_{\tau}(s,a) &= \lim_{h\to 0} \frac{Q^{\pi_{t+h}}_{\tau}(s,a) - Q^{\pi_t}_{\tau}(s,a)}{h} \\
&= \gamma \int_{S} \frac{d}{dt}V^{\pi_t}_{\tau}(s') P(ds'|s,a).
\end{align}Now observe that by \cite{david_fisher}[Proof of Proposition 2.6], we have
\begin{equation}
\frac{d}{dt}V^{\pi_t}_{\tau}(s) = \frac{1}{1-\gamma}\int_{S\times A} A^{\pi_t}_{\tau}(s,a) \partial_t \pi_t(da|s')d^{\pi_t}(ds'|s).
\end{equation}Thus we have
\begin{align}
\frac{d}{dt}Q^{\pi_t}_{\tau}(s,a) &= \frac{\gamma}{1-\gamma} \int_{S} \left(\int_{S \times A} A^{\pi_t}_{\tau}(s'',a'')\partial_t\pi_t(da''|s'')d^{\pi_t}(ds''|s') \right)P(ds'|s,a).
\end{align}
\end{proof}

\subsection{Proof of Theorem \ref{thm:convergence_1}}
\label{sec:conv1}
\begin{proof}
Recall the performance difference Lemma (Lemma \ref{lem:performance_diff}): for all $\rho \in \mathcal{P}(S)$ and $\pi,\pi'\in \Pi_{\mu}$, 
\begin{align}
&V^{\pi}_\tau(\rho)-V^{\pi'}_\tau(\rho) \\
&\quad = \frac{1}{1-\gamma}\int_S \bigg[\int_A\left(Q^{\pi'}_{\tau}(s,a)+\tau \ln \frac{d \pi'}{d\mu}(a,s)\right)(\pi-\pi')(da|s) + \tau    \operatorname{KL}(\pi(\cdot | s)|\pi'(\cdot | s)) \bigg]d^{\pi}_\rho(ds)\,.
\end{align}Now let $\pi = \pi^*$ and $\pi' = \pi_t$ and multiply both sides by $-1$ we have
\begin{align}\label{eq:performance_diff}
V^{\pi_t}_{\tau}(\rho) - V^{\pi^*}_{\tau}(\rho) &= \frac{-1}{1-\gamma}\int_{S} \Bigg(\int_A \left(Q^{\pi_t}(s,a) + \tau \ln\frac{d\pi_t}{d\mu}(a,s)\right)(\pi^* - \pi_t)(da|s) \\
&\qquad+ \tau \operatorname{KL}(\pi^*(\cdot|s)|\pi_t(\cdot|s))\Bigg)d_{\rho}^{\pi^*}(ds).
\end{align}
Recall the approximate Fisher--Rao dynamics, which we write as
\begin{equation}\label{eq:adding_zero}
\partial_t \ln \frac{d\pi_t}{d\mu}(s,a) + \left(Q_t(s,a) + \tau \ln\frac{d\pi_t}{d\mu}(a,s) - \int_{A} \left( Q_t(s,a') + \tau \ln\frac{d\pi_t}{d\mu}(a',s) \right) \pi_t(da'|s)\right) = 0.
\end{equation}Observe that since the normalisation constant (enforcing the conservation of mass along the flow) $\int_{A} \left( Q_t(s,a) + \tau \ln\frac{d\pi_t}{d\mu}(a,s) \right) \pi_t(da|s)$ is independent of $a \in A$, it holds that \[\int_{A} \left(\int_{A} \left( Q_t(s,a') + \tau \ln\frac{d\pi_t}{d\mu}(a',s) \right) \pi_t(da'|s)\right) (\pi^* - \pi_t)(da|s) = 0.\] Hence adding 0 in the form of \eqref{eq:adding_zero} into \eqref{eq:performance_diff} it holds that for all $t \geq 0$
\begin{align}\label{eq:step_1_of_conv_1}
&V^{\pi_t}_{\tau}(\rho) - V^{\pi^*}_{\tau}(\rho)= \frac{1}{1-\gamma}\Bigg(
\int_{S \times A} \partial_t \ln \frac{d \pi_t}{d\mu}(a,s) (\pi^* - \pi_t)(da|s) d_{\rho}^{\pi^*}(ds)  \\
&\quad + \int_{S \times A} (Q_t(s,a) - Q^{\pi_t}(s,a)) (\pi^* - \pi_t)(da|s) d_{\rho}^{\pi^*}(ds) -\tau\int_S \operatorname{KL}(\pi^*(\cdot|s)|\pi_t(\cdot|s)d_{\rho}^{\pi^*}(ds)\Bigg).
\end{align}By \cite[Lemma 3.8]{kerimkulov2024mirrordescentstochasticcontrol} and Corollary \ref{corr:all_gamma_KL}, for any fixed $\nu \in \Pi_{\mu}$, the map $t \to \operatorname{KL}(\nu|\pi_t)$ is differentiable. Hence we have
\begin{align}
\int_{A} \partial_t \ln \frac{d\pi_t}{d\mu}(s,a)(\pi^* - \pi_t)(da|s) &= \int_{A} \partial_t \ln \frac{d\pi_t}{d\mu}(s,a) \pi^*(da|s) - \int_{A} \partial_t \ln \frac{d\pi_t}{d\mu}(s,a)\pi_t(da|s) \\
&=  \int_{A} \partial_t \ln \frac{d\pi_t}{d\mu}(s,a) \pi^*(da|s) \\
&=-\frac{d}{dt}\operatorname{KL}(\pi^*(\cdot|s)|\pi_t(\cdot|s)),
\end{align}where we used the conservation of mass of the policy dynamics in the second equality. Substituting this into \eqref{eq:step_1_of_conv_1} we have
\begin{align}\label{eq:step_2_of_conv_1}
&V^{\pi_t}_{\tau}(\rho) - V^{\pi^*}_{\tau}(\rho)= \frac{1}{1-\gamma}\Bigg(-\frac{d}{dt}\int_{S}\operatorname{KL}(\pi^*(\cdot|s)|\pi_t(\cdot|s)) d_{\rho}^{\pi^*}(ds)\\
&+ \int_{S \times A} (Q_t(s,a) - Q^{\pi_t}(s,a)) (\pi^* - \pi_t)(da|s) d_{\rho}^{\pi^*}(ds) -\tau\int_S \operatorname{KL}(\pi^*(\cdot|s)|\pi_t(\cdot|s)d_{\rho}^{\pi^*}(ds)\Bigg).
\end{align}Focusing on the second term, we have
\begin{align}
&\int_{S \times A} (Q_t(s,a) - Q^{\pi_t}(s,a)) (\pi^* - \pi_t)(da|s) d_{\rho}^{\pi^*}(ds) \\
&\leq \left|Q_t(s,a) - Q^{\pi_t}(s,a) \right|_{B_b(S \times A)} \int_{S} \mathrm{TV}(\pi^*(\cdot|s), \pi_t(\cdot|s)) d_{\rho}^{\pi^*}(ds) \\
&\leq \frac{1}{\sqrt{2}}|\theta_t - \theta_{\pi_t}| \int_{S} \mathrm{KL}(\pi^*(\cdot|s) |\pi_t(\cdot|s))^{\frac{1}{2}}d_{\rho}^{\pi^*}(ds) \\
&\leq \frac{1}{\sqrt{2}}|\theta_t - \theta_{\pi_t}|\left( \int_{S} \mathrm{KL}(\pi^*(\cdot|s) |\pi_t(\cdot|s))d_{\rho}^{\pi^*}(ds) \right)^{\frac{1}{2}},
\end{align}where we used Pinsker's Inequality in the second inequality and Hölder's inequality in the final inequality. Now applying Young's inequality, there exists $\epsilon > 0$ such that
\begin{equation}
|\theta_t - \theta_{\pi_t}|\left( \int_{S} \mathrm{KL}(\pi^*(\cdot|s) |\pi_t(\cdot|s))d_{\rho}^{\pi^*}(ds) \right)^{\frac{1}{2}} \leq \frac{1}{2\epsilon}|\theta_t - \theta_{\pi_t}|^2 + \frac{\epsilon}{2}\int_{S} \mathrm{KL}(\pi^*(\cdot|s) |\pi_t(\cdot|s))d_{\rho}^{\pi^*}(ds).
\end{equation}Substituting this back into \eqref{eq:step_2_of_conv_1} and choosing $\epsilon = \sqrt{2}\tau$ we have
\begin{align}
V^{\pi_t}_{\tau}(\rho) - V^{\pi^*}_{\tau}(\rho)&= \frac{1}{1-\gamma}\Bigg(-\frac{d}{dt}\int_{S}\operatorname{KL}(\pi^*(\cdot|s)|\pi_t(\cdot|s)) d_{\rho}^{\pi^*}(ds)\\
&\qquad-\frac{\tau}{2}\int_S \operatorname{KL}(\pi^*(\cdot|s)|\pi_t(\cdot|s)d_{\rho}^{\pi^*}(ds) + \frac{1}{4\tau}|\theta_t - \theta_{\pi_t}|^2\Bigg).
\end{align}

Rearranging, we arrive at
\begin{align}
&\frac{d}{dt}\int_{S}\operatorname{KL}(\pi^*(\cdot|s)|\pi_t(\cdot|s))d_{\rho}^{\pi^*}(ds) \leq  - \frac{\tau}{2} \int_{S}\operatorname{KL}(\pi^*(\cdot|s)|\pi_t(\cdot|s))d_{\rho}^{\pi^*}(ds) \\
&\qquad-(1-\gamma)\left( V^{\pi_t}_{\tau}(\rho) - V^{\pi^*}_{\tau}(\rho) \right) + \frac{1}{4\tau}|\theta_t - \theta_{\pi_t}|^2.
\end{align}Applying Duhamel's principle yields
\begin{align}
&\int_{S}\operatorname{KL}(\pi^*(\cdot|s)|\pi_t(\cdot|s))d_{\rho}^{\pi^*}(ds) 
\leq e^{-\frac{\tau}{2} t} \int_{S}\operatorname{KL}(\pi^*(\cdot|s)|\pi_0(\cdot|s))d_{\rho}^{\pi^*}(ds) \\
&\qquad -(1-\gamma)\int_0^t e^{-\frac{\tau}{2}(t-r)} (V^{\pi_r}_{\tau}(\rho) - V^{\pi^*}_{\tau}(\rho)) dr +\frac{1}{2\tau}\int_0^t e^{-\frac{\tau}{2}(t-r)} |\theta_r - \theta_{\pi_r}|^2 dr.
\end{align}Now using that $\int_0^t e^{-\frac{\tau}{2}(t-r)} dr = \frac{2(1-e^{-\frac{\tau}{2}})}{\tau}$, we have
\begin{align}
&\int_{S}\operatorname{KL}(\pi^*(\cdot|s)|\pi_t(\cdot|s))d_{\rho}^{\pi^*}(ds)
\leq e^{-\frac{\tau}{2} t} \int_{S}\operatorname{KL}(\pi^*(\cdot|s)|\pi_0(\cdot|s))d_{\rho}^{\pi^*}(ds) \\ 
&\qquad - \frac{2(1-\gamma)(1-e^{-\frac{\tau}{2}})}{\tau} \min_{r \in [0,t]}\left( V^{\pi_r}_{\tau}(\rho)- V^{\pi^*}_{\tau}(\rho) \right) +\frac{1}{2\tau}\int_0^t e^{-\frac{\tau}{2}(t-r)} |\theta_r - \theta_{\pi_r}|^2 dr.
\end{align}Rearranging, we have
\begin{align}
\min_{r \in [0,t]} V^{\pi_r}_{\tau}(\rho)- V^{\pi^*}_{\tau}(\rho) &\leq \frac{\tau}{2(1-\gamma)(1-e^{-\frac{\tau}{2}})} \Bigg(e^{-\frac{\tau}{2} t} \int_{S}\operatorname{KL}(\pi^*(\cdot|s)|\pi_0(\cdot|s))d_{\rho}^{\pi^*}(ds) \\
&\qquad+ \frac{1}{2\tau}\int_0^t e^{-\frac{\tau}{2}(t-r)} |\theta_r - \theta_{\pi_r}|^2 dr \Bigg).
\end{align}which concludes the proof.

\end{proof}

\subsection{Proof of Theorem \ref{thm:main_thm}}
\label{sec:main_thm_proof}

\begin{proof}
\label{sec:conv_small_gamma}
Using the chain rule and the critic dynamics in \eqref{eq:dynamics_theta}, we have that for all $r \geq 0$
\begin{align}
\frac{1}{2\eta_r}\frac{d}{dr}|\theta_r - \theta_{\pi_r} |^2 &= \frac{1}{\eta_r}\left(\inner{\frac{d\theta_r}{dr}}{\theta_r - \theta_{\pi_r}} - \inner{\frac{d\theta_{\pi_r}}{dr}}{\theta_r - \theta_{\pi_r}}\right) \\
&=-\inner{g(\theta_r,\pi_r)}{\theta_r - \theta_{\pi_r}} - \frac{1}{\eta_r}\inner{\frac{d\theta_{\pi_r}}{dr}}{\theta_r - \theta_{\pi_r}}
\end{align}Let $\Gamma = \lambda_{\beta}(1-\gamma)(1-\sqrt{\gamma})$. Using Lemma \ref{lemma:gradient} and the $\lambda_{\beta}$-strong convexity of $L(\cdot,\pi;\beta)$ and recalling that $L(\theta_{\pi_r},\pi_r) = 0$ for all $r \geq 0$, it holds for all $r \geq 0$ that
\begin{align}
\frac{1}{2\eta_t}\frac{d}{dt}|\theta_t - \theta_{\pi_t} |^2 &= -\inner{g(\theta_t,\pi_t)}{\theta_t - \theta_{\pi_t}} - \frac{1}{\eta_t}\inner{\frac{d\theta_{\pi_t}}{dt}}{\theta_t - \theta_{\pi_t}} \\
& \leq -(1-\gamma)(1-\sqrt{\gamma})\inner{\nabla_{\theta}L(\theta_t,\pi_t;\beta
)}{\theta_t - \theta_{\pi_t}}- \frac{1}{\eta_t}\inner{\frac{d\theta_{\pi_t}}{dt}}{\theta_t - \theta_{\pi_t}} \\
&\leq -(1-\gamma)(1-\sqrt{\gamma})L(\theta_t,\pi_t;\beta) - \frac{\Gamma}{2}|\theta_t - \theta_{\pi_t} |^2  - \frac{1}{\eta_t}\inner{\frac{d\theta_{\pi_t}}{dt}}{\theta_t - \theta_{\pi_t}} \\ \label{eq:used_CS_youngs}
&\leq -(1-\gamma)(1-\sqrt{\gamma})L(\theta_t,\pi_t;\beta) - \frac{\Gamma}{2}|\theta_t - \theta_{\pi_t} |^2 + \frac{1}{2\eta_t}\left(\left|\frac{d\theta_{\pi_t} }{dt}\right|^2 + |\theta_t - \theta_{\pi_t} |^2\right) \\
&=-(1-\gamma)(1-\sqrt{\gamma})L(\theta_t,\pi_t;\beta) -\left(\frac{\Gamma}{2} - \frac{1}{2\eta_t} \right)|\theta_t - \theta_{\pi_t} |^2 + \frac{1}{2\eta_t}\left|\frac{d\theta_{\pi_t} }{dt}\right|^2,
\end{align} where we used Hölder's and Young's inequalities in \eqref{eq:used_CS_youngs}. Since $\eta_0 > \frac{1}{\Gamma}$ and $\eta_t$ is a non-decreasing function, it holds that $\eta_t > \frac{1}{\Gamma}$ for all $ t \geq 0$. Hence $\frac{\Gamma}{2} - \frac{1}{2\eta_t} > 0$ and thus we can drop the second term. Moreover the $\lambda_{\beta}$-strong convexity of $L(\cdot,\pi;\beta)$ along with $L(\theta_{\pi},\pi;\beta) = 0$ and $\nabla_{\theta} L(\theta_{\pi},\pi) = 0$ for all $ \pi \in \Pi_{\mu}$ gives that \[ |\theta_t - \theta_{\pi_t}|^2 \leq \frac{2}{\lambda_{\beta} }L(\theta_t,\pi_t;\beta).\] Hence for all $ r \geq 0$ we arrive at
\begin{align}
\frac{1}{2\eta_r}\frac{d}{dr}|\theta_r - \theta_{\pi_r} |^2 \leq -\frac{\Gamma}{2} |\theta_r - \theta_{\pi_r}|^2 + \frac{1}{2\eta_r}\left|\frac{d\theta_{\pi_r} }{dr}\right|^2.
\end{align} Rearranging, multiplying by $e^{-\tau(t-r)}$ and integrating over $r$ from $0$ to $t$, it holds for all $t \geq 0$ that
\begin{align}
\int_0^t e^{-\frac{\tau}{2}(t-r)} |\theta_r - \theta_{\pi_r}|^2 dr &\leq -\frac{1}{\Gamma} \int_0^t e^{-\frac{\tau}{2}(t-r)}\frac{1}{\eta_r}\frac{d}{dr}|\theta_r - \theta_{\pi_r}|^2 dr + \frac{1}{\Gamma} \int_0^t e^{-\frac{\tau}{2}(t-r)}\frac{1}{\eta_r}\left|\frac{d\theta_{\pi_r} }{dt}\right|^2 dr.
\end{align} Integrating the first term by parts (identically to \eqref{eq:int_by_parts} from the proof of Theorem \ref{Theorem:gronwall_KL}), we have
\begin{align}
&\int_0^t e^{-\frac{\tau}{2}(t-r)} |\theta_r - \theta_{\pi_r}|^2 dr \leq \frac{1}{\Gamma}\Bigg(-\frac{|\theta_t - \theta_{\pi_t}|^2}{\eta_t} + e^{-\frac{\tau}{2}t}\frac{|\theta_0 - \theta_{\pi_0}|^2}{\eta_0} \\
&\qquad+ \frac{\tau}{2} \int_0^t e^{-\frac{\tau}{2}(t-r)}\frac{1}{\eta_r}|\theta_r - \theta_{\pi_r}|^2 dr - \int_0^t |\theta_r - \theta_{\pi_r}|^2 \frac{e^{-\frac{\tau}{2}(t-r)}\frac{d}{dr}\eta_r}{\eta_r^2} dr \\
&\qquad+\int_0^t e^{-\frac{\tau}{2}(t-r)}\frac{1}{\eta_r}\left|\frac{d\theta_{\pi_r} }{dr}\right|^2 dr \Bigg).
\end{align}

Since for all $t\geq 0$ it holds that $\eta_t \geq 1$ and $\frac{d}{dt}\eta_t \geq 0$, we have that\[\int_0^t |\theta_r - \theta_{\pi_r}|^2 \frac{e^{-\frac{\tau}{2}(t-r)}\frac{d}{dr}\eta_r}{\eta_r^2} dr \geq 0.\] Thus after dropping all negative terms and using that $\eta_t \geq \eta_0$ for all $t \geq 0$, we have
\begin{equation}
\left(1-\frac{\tau}{2\Gamma \eta_0}\right)\int_0^t e^{-\frac{\tau}{2}(t-r)} |\theta_r - \theta_{\pi_r}|^2 dr \leq e^{-\frac{\tau}{2}}\frac{|\theta_0 - \theta_{\pi_0}|^2}{\Gamma \eta_0} + \int_0^t e^{-\frac{\tau}{2}(t-r)}\frac{1}{\eta_r}\left|\frac{d\theta_{\pi_r} }{dr}\right|^2 dr.
\end{equation}Since $\eta_0 > \frac{1}{2\Gamma}$ and $\tau < 1$, it holds that $1-\frac{\tau}{2\Gamma \eta_0} > 0$ and hence it holds that
\begin{align}
\int_0^t e^{-\frac{\tau}{2}(t-r)} |\theta_r - \theta_{\pi_r}|^2 dr \leq e^{-\frac{\tau}{2}}\frac{|\theta_0 - \theta_{\pi_0}|^2}{\Gamma \eta_0\left(1-\frac{\tau}{2\Gamma \eta_0}\right)} + \frac{1}{\left(1-\frac{\tau}{2\Gamma \eta_0}\right)}\int_0^t e^{-\frac{\tau}{2}(t-r)}\frac{1}{\eta_r}\left|\frac{d\theta_{\pi_r} }{dr}\right|^2 dr,
\end{align}which concludes the proof.
\end{proof}
\subsection{Proof of Theorem \ref{thm:all_gamma_conv}}
\label{sec:proof_of_all_gamma_conv}
\begin{proof}

By Theorem \ref{thm:main_thm}, we have
\begin{align}
\int_0^t e^{-\frac{\tau}{2}(t-r)} |\theta_r - \theta_{\pi_r}|^2 dr \leq e^{-\frac{\tau}{2}}\frac{|\theta_0 - \theta_{\pi_0}|^2}{\Gamma \eta_0\left(1-\frac{\tau}{2\Gamma \eta_0}\right)} + \frac{1}{\left(1-\frac{\tau}{2\Gamma \eta_0}\right)}\int_0^t e^{-\frac{\tau}{2}(t-r)}\frac{1}{\eta_r}\left|\frac{d\theta_{\pi_r} }{dr}\right|^2 dr.
\end{align} Hence it remains to characterise the growth of the final integral. Observe that for all $\pi \in \mathcal{P}(A|S)$, $\theta_{\pi} \in \mathbb{R}^N$ satisfies the least-squares optimality condition given by
\begin{equation}
\theta_{\pi} = \argmin_{\theta} L(\theta,\pi;\beta) = \left(\int_{S \times A} \phi(s,a)\phi(s,a)^{\top}\beta(da,ds) \right)^{-1}\left(\int_{S \times A} \phi(s,a) Q^{\pi}_{\tau}(s,a) \, \beta(ds,da)\right).
\end{equation}Setting $\pi = \pi_t$ and differentiating time we arrive at
\begin{equation}
\frac{d\theta_{\pi_t}}{dt} = \left(\int_{S \times A} \phi(s,a)\phi(s,a)^{\top}\beta(da,ds) \right)^{-1}\left(\int_{S \times A} \phi(s,a) \frac{d}{dt}Q^{\pi_t}(s,a) \, \beta(ds,da)\right).
\end{equation}Hence by Lemma \ref{lemma:Q_derivative}, Assumption \ref{as:bounded_phi} and Assumption \ref{as:e_value}, for all $t \geq 0$ it holds that
\begin{align}\label{eq:bounded_rate_of_change}
\left| \frac{d\theta_{\pi_t}}{dt} \right| 
&= \left| \left( \int_{S \times A} \phi(s,a)\phi(s,a)^{\top} \beta(da,ds) \right)^{-1}
\left( \int_{S \times A} \phi(s,a) \frac{d}{dt} Q^{\pi_t}_{\tau}(s,a) \, \beta(ds,da) \right) \right| \\
&\leq  \left| \left( \int_{S \times A} \phi(s,a)\phi(s,a)^{\top} \beta(da,ds) \right)^{-1}\right|_{\mathrm{op}} \left| \frac{d}{dt} Q^{\pi_t}_{\tau} \right|_{B_b(S \times A)} \\
&= \frac{1}{\lambda_{\beta}} \left| \frac{d}{dt} Q^{\pi_t} \right|_{B_b(S \times A)} \\
&= \frac{\gamma}{\lambda_{\beta}(1 - \gamma)} \left| \int_{S} 
\left( \int_{S \times A} A^{\pi_t}_{\tau}(s'',a'') 
\partial_t \pi_t(da''|s'') \, d^{\pi_t}(ds''|s') \right) 
P(ds'|\cdot,\cdot) \right|_{B_b(S \times A)} \\
&\leq \frac{\gamma}{\lambda_{\beta}(1 - \gamma)} 
\left| A^{\pi_t}_{\tau} \right|_{B_b(S \times A)} 
\sup_{s \in S} \left| \partial_t \pi_t(\cdot|s) \right|_{\mathcal{M}(A)}.
\end{align}Now using Lemma \ref{lemma:bounds_along_flow}, it holds that
\begin{align}
&\left| A^{\pi_t}_{\tau} \right|_{B_b(S \times A)} 
\sup_{s \in S} \left| \partial_t \pi_t(\cdot|s) \right|_{\mathcal{M}(A)} \leq\left| A^{\pi_t}_{\tau} \right|_{B_b(S \times A)}\left| A_t\right|_{B_b(S \times A)}\\
&\leq \left( 2\left|Q^{\pi_t}_\tau\right|_{B_b(S\times A)} + 2\tau \left|\ln \frac{d\pi_t}{d\mu} \right|_{B_b(S\times A)} \right)\left( 2\left|Q_t\right|_{B_b(S\times A)} + 2\tau \left|\ln \frac{d\pi_t}{d\mu} \right|_{B_b(S\times A)} \right).
\end{align}  Hence by Corollaries \ref{corr:all_gamma_KL} and \ref{corr:all_gamma_theta} and Lemma \ref{lemma:bounds_along_flow}, there exists $\alpha_1,\alpha_2 > 0$ such that \[\left| \frac{d\theta_{\pi_t} }{dt}\right|^2 \leq \alpha_1 e^{\alpha_2 t}.\] Thus Theorem \ref{thm:main_thm} becomes
\begin{align}
\int_0^t e^{-\frac{\tau}{2}(t-r)} |\theta_r - \theta_{\pi_r}|^2 dr \leq e^{-\frac{\tau}{2}}\frac{|\theta_0 - \theta_{\pi_0}|^2}{\Gamma \eta_0\left(1-\frac{\tau}{2\Gamma \eta_0}\right)} + \frac{\alpha_1}{\left(1-\frac{\tau}{2\Gamma \eta_0}\right)}\int_0^t e^{-\frac{\tau}{2}(t-r)}\frac{e^{\alpha_2 r}}{\eta_r}dr.
\end{align}Let $\eta_t = \eta_0 e^{k_1 t}$ for any $k_1 > \frac{\tau}{2} + \alpha_2$. Then observe that 
\begin{align}
\int_0^t e^{-\frac{\tau}{2}(t-r)}\frac{e^{\alpha_2 r}}{\eta_r}dr &= \frac{1}{\eta_0}e^{-\frac{\tau}{2}t} \int_0^t e^{\left(\frac{\tau}{2} + \alpha_2 - k_1\right)r} dr \\
&\leq \frac{1}{\eta_0}e^{-\frac{\tau}{2}t} \left(\frac{e^{\left(\frac{\tau}{2} + \alpha_2 - k_1\right)t}-1}{\frac{\tau}{2} + \alpha_2 - k_1} \right) \\
&\leq  \frac{e^{-\frac{\tau}{2}t}}{\eta_0\left(\frac{\tau}{2} + \alpha_2 - k_1\right)},
\end{align}
hence all together it holds that
\begin{align}
\int_0^t e^{-\frac{\tau}{2}(t-r)} |\theta_r - \theta_{\pi_r}|^2 dr \leq e^{-\frac{\tau}{2}}\frac{|\theta_0 - \theta_{\pi_0}|^2}{\Gamma \eta_0\left(1-\frac{\tau}{2\Gamma \eta_0}\right)} + e^{-\frac{\tau}{2}t}\frac{\alpha_1}{\left(\eta_0-\frac{\tau}{2\Gamma}\right)\left(\frac{\tau}{2} + \alpha_2 - k_1 \right)}.
\end{align} Substituting this into the result from Theorem \ref{thm:main_thm} concludes the proof.
\end{proof}

\section{Additional results}\label{sec:small_gamma}

\begin{corollary}[Uniform boundedness]\label{corr:small_gamma_KL}Under the same assumptions as Theorem \ref{Theorem:gronwall_KL}, for $\gamma \in (0,1)$ such that $\frac{64 \gamma^2 }{\Gamma^2-\frac{\Gamma\tau}{\eta_0} } < 1$ it holds that $a_2 < \tau$ and for all $t \geq 0$ it holds that
\begin{equation}
\operatorname{KL}(\pi_t(\cdot|s)|\mu)^2\leq\frac{a_1\tau}{\tau - a_2}
\end{equation}
\end{corollary} 

\subsection{Proof of Corollary \ref{corr:small_gamma_KL}}
\label{sec:small_gamma_KL}
\begin{proof}
By Theorem \ref{Theorem:gronwall_KL} we have that
\begin{equation}
\mathrm{K
}_t^2 \leq  a_1 + a_2\int_0^t e^{-\tau(t-r)}\mathrm{K}_r^2 dr.
\end{equation}Taking the supremum over $[0,t]$ on the right hand side, we have
\begin{equation}
\mathrm{K
}_t^2 \leq  a_1 + \frac{a_2}{\tau}\sup_{r \in [0,t]} \mathrm{K}_r^2.
\end{equation} Since this holds for all $ t \geq 0$, we have
\begin{equation}
\sup_{r \in [0,t]} \mathrm{K}_r^2 \leq a_1 + \frac{a_2}{\tau} \sup_{r \in [0,t]} \mathrm{K}_r^2.
\end{equation} Now forcing $1-\frac{a_2}{\tau} > 0$, which is equivalent to the condition\[\frac{64 \gamma^2 }{\Gamma^2-\frac{\Gamma\tau}{\eta_0} } < 1.\] Hence after rearranging we have
\begin{equation}
\mathrm{K}_t^2 \leq  \sup_{r \in [0,t]} \mathrm{K}_r^2 \leq \frac{a_1\tau}{\tau - a_2}
\end{equation}
\end{proof}
\begin{remark}\label{remark:tautoinf}
    Observe that if one does not apply the loose upper bound $e^{-\tau t}\leq 1$ in \eqref{eq:duhamel} from the proof of Theorem \ref{Theorem:gronwall_KL}, it holds that 
    \begin{equation}
        a_1 = a_1(t) = 8e^{-2\tau t}(C_1)^2 + \frac{32}{\tau}\sigma_1
    \end{equation} with $\sigma_1 : = \frac{|\theta_0|^2}{\Gamma \eta_0 \left(1-\frac{\tau}{\Gamma\eta_0} \right)} + \frac{2 |c|_{B_b(S \times A)}^2}{\Gamma^2 \tau \left( 1-\frac{\tau}{\Gamma\eta_0}\right)}$. Then choosing $\eta_0 = \tau + \epsilon$ for any $\epsilon > 0$ so that the conditions of Theorem \ref{Theorem:gronwall_KL} holds, formally sending $\tau \to \infty$ we obtain $\operatorname{KL}(\pi_t(\cdot|s)|\mu) \to 0$ for all $s \in S$.
\end{remark}

\begin{corollary}\label{cor:small_gamma_theta}Under the conditions of Corollary \ref{corr:small_gamma_KL}, there exists $R > 0$ such that for all $t \geq 0$ it holds that
\begin{equation}
|\theta_t| \leq R
\end{equation}
\end{corollary} 

\subsection{Proof of Corollary \ref{cor:small_gamma_theta}}
\label{sec:small_gamma_theta}
\begin{proof}
By Corollary \ref{corr:small_gamma_KL}, for sufficiently small $\gamma > 0$ it holds that for all $t \geq 0$, \[\mathrm{K}_t^2 \leq \frac{a_1\tau}{\tau - a_2}.\] Hence by Lemma \ref{lemma:lyap_drift} we have
\begin{align}
\frac{1}{2}\frac{d}{dt}\left|\theta_t \right|^2 &\leq -\eta_t\frac{\Gamma}{2} |\theta_t|^2 + \eta_t\left(\frac{2|c|_{B_b(S \times A)}^2 + 2\tau^2\gamma^2 \left(\frac{a_1\tau}{\tau - a_2}\right)}{\Gamma^2}\right).
\end{align} The uniform boundedness in time of $|\theta_t|$ then follows by Gr\"onwall's Lemma (Lemma \ref{lemma:gronwall}).
\end{proof}

\begin{corollary}\label{thm:small_gamma_regime}
Under the same assumptions as Theorem \ref{thm:main_thm}, for $\gamma \in (0,1)$ such that $\frac{2\sqrt{2}\gamma}{\sqrt{\Gamma^2-\frac{\Gamma\tau}{\eta_0} }} < 1$ there exists $d_1 > 0$ such that for all $t \geq 0$,
\begin{align}
\min_{r \in [0,t]} V^{\pi_r}_{\tau}(\rho)- V^{\pi^*}_{\tau}(\rho) &\leq \frac{\tau}{2(1-\gamma)(1-e^{-\frac{\tau}{2}t})} \Bigg(e^{-\frac{\tau}{2} t} \int_{S}\operatorname{KL}(\pi^*(\cdot|s)|\pi_0(\cdot|s))d_{\rho}^{\pi^*}(ds) \\
&\qquad+ d_1\int_0^t e^{-\frac{\tau}{2}(t-r)}\frac{1}{\eta_r} dr \Bigg).
\end{align}
\end{corollary}

\subsection{Proof of Theorem \ref{thm:all_gamma_conv}}
\begin{proof}

Following completely identically to the proof of Theorem \ref{thm:all_gamma_conv}, we have
\begin{align}
\left| \frac{d\theta_{\pi_t}}{dt} \right|  &\leq \frac{\gamma}{\lambda_{\beta}(1 - \gamma)} 
\left| A^{\pi_t}_{\tau} \right|_{B_b(S \times A)} 
\sup_{s \in S} \left| \partial_t \pi_t(\cdot|s) \right|_{\mathcal{M}(A)} \\
&\leq \frac{4}{(1-\gamma)^2}\left(\left|c \right|_{B_b(S\times A)} + \mathrm{K}_t \right)^2 + 4\tau \left(C_1 + \frac{2}{\tau} \sup_{r \in [0,t]} |\theta_r| + \sup_{r \in [0,t]}K_r \right)^2.
\end{align}Then by Corollaries \ref{corr:small_gamma_KL} and \ref{cor:small_gamma_theta}, there exists $b_2 > 0$ such that $ \left| \frac{d\theta_{\pi_t}}{dt} \right|^2 \leq d_1$. Hence by Theorem \ref{thm:main_thm} we have
\begin{align}
\min_{r \in [0,t]} V^{\pi_r}_{\tau}(\rho)- V^{\pi^*}_{\tau}(\rho) &\leq \frac{\tau}{2(1-\gamma)(1-e^{-\frac{\tau}{2}})} \Bigg(e^{-\frac{\tau}{2} t} \Bigg(\int_{S}\operatorname{KL}(\pi^*(\cdot|s)|\pi_0(\cdot|s))d_{\rho}^{\pi^*}(ds) \\
&\qquad+ d_1\int_0^t e^{-\frac{\tau}{2}(t-r)}\frac{1}{\eta_r} dr \Bigg)\,.
\end{align}
\end{proof}

\section*{Acknowledgements}
DZ was supported by the EPSRC Centre for Doctoral Training in Mathematical Modelling, Analysis and Computation (MAC-MIGS) funded by the UK Engineering and Physical Sciences Research Council (grant EP/S023291/1), Heriot-Watt University and the University of Edinburgh. 
The work on this project by D\v{S} was partially supported by a grant from the Simons Foundation.
D\v{S} and LS acknowledge funding from the UKRI Prosperity Partnerships grant APP43592: AI$^2$ - Assurance and Insurance for Artificial Intelligence, which supported this work.
The authors would like to thank the Isaac Newton Institute for Mathematical Sciences, Cambridge, for support and hospitality during the programme ``Bridging Stochastic Control And Reinforcement Learning'', where work on this paper was undertaken. 
This work was supported by EPSRC grant EP/V521929/1.

Last but not least, the authors would like to thank the anonymous reviewers whose insightful comments have helped improve the paper.

\bibliographystyle{plain}
\bibliography{iclr2026_conference}

@book{bertsekas1996stochastic,
  title={Stochastic optimal control: the discrete-time case},
  author={Bertsekas, Dimitri and Shreve, Steven E},
  volume={5},
  year={1996},
  publisher={Athena Scientific}
}

@inproceedings{mei2020global,
  title={On the global convergence rates of softmax policy gradient methods},
  author={Mei, Jincheng and Xiao, Chenjun and Szepesvari, Csaba and Schuurmans, Dale},
  booktitle={International conference on machine learning},
  pages={6820--6829},
  year={2020},
  organization={PMLR}
}

@article{lan2023policy,
  title={Policy mirror descent for reinforcement learning: Linear convergence, new sampling complexity, and generalized problem classes},
  author={Lan, Guanghui},
  journal={Mathematical programming},
  volume={198},
  number={1},
  pages={1059--1106},
  year={2023},
  publisher={Springer}
}

@article{schulman2017proximal,
  title={Proximal policy optimization algorithms},
  author={Schulman, John and Wolski, Filip and Dhariwal, Prafulla and Radford, Alec and Klimov, Oleg},
  journal={arXiv preprint arXiv:1707.06347},
  year={2017}
}

@inproceedings{schulman2015trust,
  title={Trust region policy optimization},
  author={Schulman, John and Levine, Sergey and Abbeel, Pieter and Jordan, Michael and Moritz, Philipp},
  booktitle={International conference on machine learning},
  pages={1889--1897},
  year={2015},
  organization={PMLR}
}

@inproceedings{jordan,
  author    = {Yufeng Zhang and Siyu Chen and Zhuoran Yang and Michael Jordan and Zhaoran Wang},
  title     = {Wasserstein Flow Meets Replicator Dynamics: A Mean-Field Analysis of Representation Learning in Actor-Critic},
  booktitle = {Advances in Neural Information Processing Systems},
  year      = {2021},
  url       = {https://openreview.net/forum?id=9IJLHPuLpvZ}
}

@article{david_fisher,
  author  = {Kerimkulov, Bekzhan and Leahy, James-Michael and \v{S}i\v{s}ka, David and Szpruch, Lukasz and Zhang, Yufei},
  title   = {A {F}isher--{R}ao Gradient Flow for Entropy-Regularised {M}arkov Decision Processes in {P}olish Spaces},
  journal = {Foundations of Computational Mathematics},
  year    = {2025},
  doi     = {10.1007/s10208-025-09729-3},
  url     = {https://doi.org/10.1007/s10208-025-09729-3},
  issn    = {1615-3383}
}

@article{cayci,
  author  = {Cayci, Semih and He, Niao and Srikant, R.},
  title   = {Convergence of Entropy-Regularized Natural Policy Gradient with Linear Function Approximation},
  journal = {SIAM Journal on Optimization},
  volume  = {34},
  number  = {3},
  pages   = {2729--2755},
  year    = {2024},
  doi     = {10.1137/22M1540156},
  url     = {https://doi.org/10.1137/22M1540156}
}

@misc{lan,
  author        = {Caleb Ju and Guanghui Lan},
  title         = {Policy Optimization over General State and Action Spaces},
  year          = {2024},
  eprint        = {2211.16715},
  archivePrefix = {arXiv},
  primaryClass  = {cs.LG},
  url           = {https://arxiv.org/abs/2211.16715}
}

@article{stochastic_approx_application_AC,
  author  = {Hong, Mingyi and Wai, Hoi-To and Wang, Zhaoran and Yang, Zhuoran},
  title   = {A Two-Timescale Stochastic Algorithm Framework for Bilevel Optimization: Complexity Analysis and Application to Actor-Critic},
  journal = {SIAM Journal on Optimization},
  volume  = {33},
  number  = {1},
  pages   = {147--180},
  year    = {2023},
  doi     = {10.1137/20M1387341},
  url     = {https://doi.org/10.1137/20M1387341}
}

@article{polyak,
  author  = {Liu, L. and Majka, M. B. and Szpruch, Ł.},
  title   = {Polyak--Łojasiewicz Inequality on the Space of Measures and Convergence of Mean-Field Birth-Death Processes},
  journal = {Applied Mathematics and Optimization},
  volume  = {87},
  pages   = {48},
  year    = {2023},
  doi     = {10.1007/s00245-022-09962-0},
  url     = {https://doi.org/10.1007/s00245-022-09962-0}
}

@article{borkar1997actor,
  title={The actor-critic algorithm as multi-time-scale stochastic approximation},
  author={Borkar, V. S. and Konda, V. R.},
  journal={Sadhana},
  volume={22},
  number={5},
  pages={525--543},
  year={1997},
  publisher={Springer},
  doi={10.1007/BF02745577}
}

@book{dupuis1997weak,
  author    = {Paul Dupuis and Richard S. Ellis},
  title     = {A Weak Convergence Approach to the Theory of Large Deviations},
  year      = {1997},
  publisher = {John Wiley \& Sons, Inc.},
  isbn      = {9780471076728},
  doi       = {10.1002/9781118165904},
  series    = {Wiley Series in Probability and Statistics},
  url       = {https://doi.org/10.1002/9781118165904}
}

@inproceedings{konda,
 author = {Konda, Vijay and Tsitsiklis, John},
 booktitle = {Advances in Neural Information Processing Systems},
 editor = {S. Solla and T. Leen and K. M\"{u}ller},
 pages = {},
 publisher = {MIT Press},
 title = {Actor-Critic Algorithms},
 url = {https://proceedings.neurips.cc/paper_files/paper/1999/file/6449f44a102fde848669bdd9eb6b76fa-Paper.pdf},
 volume = {12},
 year = {1999}
}

@InProceedings{barakat22a-linear,
  title = 	 { Analysis of a Target-Based Actor-Critic Algorithm with Linear Function Approximation },
  author =       {Barakat, Anas and Bianchi, Pascal and Lehmann, Julien},
  booktitle = 	 {Proceedings of The 25th International Conference on Artificial Intelligence and Statistics},
  pages = 	 {991--1040},
  year = 	 {2022},
  editor = 	 {Camps-Valls, Gustau and Ruiz, Francisco J. R. and Valera, Isabel},
  volume = 	 {151},
  series = 	 {Proceedings of Machine Learning Research},
  month = 	 {28--30 Mar},
  publisher =    {PMLR},
  url = 	 {https://proceedings.mlr.press/v151/barakat22a.html}
}

@InProceedings{linear2,
  title = 	 {Provably Convergent Two-Timescale Off-Policy Actor-Critic with Function Approximation},
  author =       {Zhang, Shangtong and Liu, Bo and Yao, Hengshuai and Whiteson, Shimon},
  booktitle = 	 {Proceedings of the 37th International Conference on Machine Learning},
  pages = 	 {11204--11213},
  year = 	 {2020},
  volume = 	 {119},
  series = 	 {Proceedings of Machine Learning Research},
  month = 	 {13--18 Jul},
  publisher =    {PMLR},
  url = 	 {https://proceedings.mlr.press/v119/zhang20s.html}
}

@article{
cayci_neural,
title={Finite-Time Analysis of Entropy-Regularized Neural Natural Actor-Critic Algorithm},
author={Semih Cayci and Niao He and R. Srikant},
journal={Transactions on Machine Learning Research},
issn={2835-8856},
year={2024},
url={https://openreview.net/forum?id=BkEqk7pS1I},
note={}
}

@inproceedings{kakade_PD,
  title={Approximately Optimal Approximate Reinforcement Learning},
  author={Sham M. Kakade and John Langford},
  booktitle={International Conference on Machine Learning},
  year={2002},
  url={https://api.semanticscholar.org/CorpusID:31442909}
}

@inproceedings{linear3,
author = {Zanette, Andrea and Wainwright, Martin J. and Brunskill, Emma},
title = {Provable benefits of actor-critic methods for offline reinforcement learning},
year = {2021},
isbn = {9781713845393},
publisher = {Curran Associates Inc.},
address = {Red Hook, NY, USA},
booktitle = {Proceedings of the 35th International Conference on Neural Information Processing Systems},
articleno = {1044},
numpages = {15},
series = {NIPS '21}
}

@article{qiu2021finite_approxerror1,
  author    = {Shuang Qiu and Zhaoran Yang and Jianfeng Ye and Zhuoran Wang},
  title     = {On Finite-Time Convergence of Actor-Critic Algorithm},
  journal   = {IEEE Journal on Selected Areas in Information Theory},
  volume    = {2},
  number    = {2},
  pages     = {652--664},
  year      = {2021},
  month     = jun,
  doi       = {10.1109/JSAIT.2021.3078754}
}

@article{lin2025rethinkingglobalconvergencesoftmax,
  title={Rethinking the global convergence of softmax policy gradient with linear function approximation},
  author={Lin, Max Qiushi and Mei, Jincheng and Aghaei, Matin and Lu, Michael and Dai, Bo and Agarwal, Alekh and Schuurmans, Dale and Szepesvari, Csaba and Vaswani, Sharan},
  journal={arXiv preprint arXiv:2505.03155},
  year={2025}
}

@article{kerimkulov2024mirrordescentstochasticcontrol,
  title = {Mirror descent for stochastic control problems with measure-valued controls},
journal = {Stochastic Processes and their Applications},
volume = {190},
pages = {104765},
year = {2025},
issn = {0304-4149},
doi = {https://doi.org/10.1016/j.spa.2025.104765},
url = {https://www.sciencedirect.com/science/article/pii/S0304414925002091},
author = {Bekzhan Kerimkulov and David \v{S}i\v{s}ka and Lukasz Szpruch and Yufei Zhang},
}

@article{sutton1988learning,
  title={Learning to predict by the methods of temporal differences},
  author={Sutton, Richard S.},
  journal={Machine Learning},
  volume={3},
  number={1},
  pages={9--44},
  year={1988},
  publisher={Springer},
  doi={10.1007/BF00115009}
}

@book{brezis,
  author    = {Haim Brezis},
  title     = {Functional Analysis, Sobolev Spaces and Partial Differential Equations},
  year      = {2011},
  publisher = {Springer},
  address   = {New York, NY},
  doi       = {10.1007/978-0-387-70914-7},
  isbn      = {978-0-387-70913-0},
  url       = {https://doi.org/10.1007/978-0-387-70914-7},
}

@inproceedings{ma2024skill_quadraptor,
  author    = {Ma, H. and Ren, Z. and Dai, B. and Li, N.},
  title     = {Skill Transfer and Discovery for Sim-to-Real Learning: A Representation-Based Viewpoint},
  booktitle = {Proceedings of the 2024 IEEE/RSJ International Conference on Intelligent Robots and Systems (IROS)},
  year      = {2024},
  address   = {Abu Dhabi, United Arab Emirates},
  pages     = {8603--8609},
  doi       = {10.1109/IROS58592.2024.10801637}
}

@inproceedings{ren2023stochastic_truncation,
  author    = {Ren, T. and Ren, Z. and Li, N. and Dai, B.},
  title     = {Stochastic Nonlinear Control via Finite-dimensional Spectral Dynamic Embedding},
  booktitle = {Proceedings of the 2023 62nd IEEE Conference on Decision and Control (CDC)},
  year      = {2023},
  address   = {Singapore, Singapore},
  pages     = {795--800},
  doi       = {10.1109/CDC49753.2023.10383842}
}

@article{bhandari_td_learning,
author = {Bhandari, Jalaj and Russo, Daniel and Singal, Raghav},
title = {A Finite Time Analysis of Temporal Difference Learning with Linear Function Approximation},
year = {2021},
issue_date = {May-June 2021},
publisher = {INFORMS},
address = {Linthicum, MD, USA},
volume = {69},
number = {3},
issn = {0030-364X},
url = {https://doi.org/10.1287/opre.2020.2024},
doi = {10.1287/opre.2020.2024},
journal = {Oper. Res.},
month = may,
pages = {950–973},
numpages = {24}
}

@book{hernandez-lerma1996discrete,
  author    = {On{\'e}simo Hern{\'a}ndez-Lerma and Jean Bernard Lasserre},
  title     = {Discrete-Time Markov Control Processes: Basic Optimality Criteria},
  series    = {Stochastic Modelling and Applied Probability},
  publisher = {Springer},
  address   = {New York, NY},
  year      = {1996},
  isbn      = {978-0-387-94579-8},
  doi       = {10.1007/978-1-4612-0729-0},
  note      = {First edition},
}

\end{document}